\documentclass{article}
\usepackage{amsmath,amssymb,amsthm,graphicx,epsfig,float,url}
\usepackage[colorlinks=true,citecolor={Plum},linkcolor={Periwinkle}]{hyperref}
\usepackage{pdfsync}
\usepackage[usenames, dvipsnames]{xcolor}
\usepackage{tikz}
\usepackage{subfig}
\usepackage{bbm}
\usepackage{stmaryrd}
\usepackage{amssymb}
\usepackage{mathrsfs}  
\usepackage{caption}
\usepackage{float}
\usepackage{esint}

\usepackage{dsfont}
\usepackage[makeroom]{cancel}
\usetikzlibrary{patterns}
\newcommand{\R}{\textnormal{I\kern-0.21emR}}
\newcommand{\N}{\textnormal{I\kern-0.21emN}}

\renewcommand{\geq}{\geqslant}
\renewcommand{\leq}{\leqslant}

\def\B{{\mathbb B}}
\def\e{{\varepsilon}}

\def\doubleunderline#1{\underline{\underline{#1}}}
\allowdisplaybreaks

\def\YYint#1#2#3{{\setbox0=\hbox{$#1{#2#3}{\iint}$}
    \vcenter{\hbox{$#2#3$}}\kern-.51\wd0}}
 
\usepackage[dvipsnames]{xcolor}

\newtheorem{theorem}{Theorem}

\newtheorem{material}{material}
\newtheorem{proposition}[material]{Proposition}

\newtheorem{definition}[material]{Definition}
\newtheorem{lemma}[material]{Lemma}

\theoremstyle{definition}\newtheorem{remark}[material]{Remark}

\def\O{{\Omega}}
\def\n{{\nabla}}
\def\p{{\varphi}}

\usepackage{xargs}
 \usepackage[colorinlistoftodos,textsize=small]{todonotes}
 \newcommandx{\unsure}[2][1=]{\todo[linecolor=red,backgroundcolor=red!25,bordercolor=red,#1]{#2}}
 \newcommandx{\change}[2][1=]{\todo[linecolor=blue,backgroundcolor=blue!25,bordercolor=blue,#1]{#2}}
 \newcommandx{\info}[2][1=]{\todo[linecolor=green,backgroundcolor=green!25,bordercolor=green,#1]{#2}}
 \newcommandx{\improvement}[2][1=]{\todo[linecolor=yellow,backgroundcolor=yellow!25,bordercolor=yellow,#1]{#2}}
 
  \newcommandx{\biblio}[2][1=]{\todo[linecolor=blue,backgroundcolor=magenta!25,bordercolor=blue,#1]{#2}}

\begin{document}
\title{Quantitative stability for eigenvalues of Schr\"{o}dinger operator, Quantitative bathtub principle\\ \&\\Application to the turnpike property for a bilinear optimal control problem}


\author{Idriss Mazari\footnote{Technische Universit\"{a}t Wien, Institute of Analysis and Scientific Computing, 8-10 Wiedner Haupstrasse, 1040 Wien (\texttt{idriss.mazari@tuwien.ac.at})},  \quad Dom\`enec Ruiz-Balet\footnote{Chair of Computational Mathematics, Fundaci\'on Deusto, Av. de las Universidades, 24,  48007 Bilbao, Basque Country, Spain} \footnote{ Departamento de Matem\'eticas, Universidad Aut\'onoma de Madrid, 28049 Madrid, Spain, (\texttt{domenec.ruiz@deusto.es })}}
\date{\today}

\maketitle

\begin{abstract}
This work is concerned with two optimisation problems that we tackle from a qualitative perspective. The first one deals with quantitative inequalities for spectral optimisation problems for Schr\"{o}dinger operators in general domains, the second one deals with the turnpike property for optimal bilinear control problems. In the first part of this article, we prove, under mild technical assumptions, quantitative inequalities for the optimisation of the first eigenvalue of $-\Delta-V$ with Dirichlet boundary conditions with respect to the potential $V$, under $L^\infty$ and $L^1$ constraints. This is done using a new method of proof which relies on in a crucial way on a quantitative bathtub principle. We believe our approach susceptible of being generalised to other steady elliptic optimisation problems. In the second part of this paper, we use this inequality to tackle a turnpike problem. Namely, considering a bilinear control system of the form $u_t-\Delta u=\mathcal V u$, $\mathcal V=\mathcal V(t,x)$ being the control, can we give qualitative information, under  $L^\infty$ and $L^1$ constraints on $\mathcal V$, on the solutions of the optimisation problem $\sup \int_\O u(T,x)dx$? We prove that the quantitative inequality for eigenvalues implies an integral turnpike property: defining $\mathcal I^*$ as the set of optimal potentials for the eigenvalue optimisation problem and $\mathcal V_T^*$ as a solution of the bilinear optimal control problem, the quantity $\int_0^T \operatorname{dist}_{L^1}(\mathcal V_T^*(t,\cdot)\,, \mathcal I^*)^2$ is bounded uniformly in $T$.

\end{abstract}


\medskip


\paragraph{Acknowledgment.}

I. Mazari was  supported by the French ANR Project ANR-18-CE40-0013 - SHAPO on Shape Optimization and by the Austrian Science Fund (FWF) through the grant I4052-N32 .

D. Ruiz-Balet was  supported by the European Research Council (ERC) under the European Union's Horizon 2020 research and innovation programme (grant agreement No. 694126-DyCon).

This project has received funding from the European Union's Horizon 2020 research and innovation programme under the Marie Sklodowska-Curie grant agreement No.765579-ConFlex, grant MTM2017-92996 of MINECO (Spain), ICON of the French ANR and "Nonlocal PDEs: Analysis, Control and Beyond", AFOSR Grant FA9550-18-1-0242 and the Alexander von Humboldt-Professorship program.


\section{Introduction}
\subsection{General setting and structure of the article}

The overall objective of this article is to establish a link between two key classes of results of optimisation and optimal control theory: \emph{quantitative estimates for optimal control problems} and the \emph{turnpike-phenomenon}. We will establish our results in the context of \emph{bilinear control problems}. We briefly present the main protagonists of the article, before laying out the structure of this paper.

\paragraph{Quantitative estimates}
For a function $J=J(u)$, where $u=u(x)$ is the (stationary) control, which has, under some contraints, a certain minimiser $u^*$,  quantitative estimates roughly amount to establishing an inequality of the form 
$$J(u)-J(u^*)\geq c ||u-u^*||_X^\alpha$$ for a certain constant $c>0$, a certain norm $X$ and a certain exponent $\alpha$. In the context of this article, since we will enforce $L^\infty$ and $L^1$ constraints on the control variable $u$, we will prove that such inequalities are linked to quantitative inequalities for shape optimisations problems, which has been a tremendously active field of research; we point, for the time being, to the seminal \cite{FuscoMaggiPratelli}. In our case, notable differences with the " shape optimisation " context will arise and call for new methods.

In this article, the spectral optimisation problem considered reads as follows: considering, for a certain potential $V=V(x)$, the first Dirichlet eigenvalue $\lambda(V)$ of 
$$-\Delta -V$$ solve the optimisation problem
\begin{equation}\label{Eq0}\tag{$\bold O_I$}\inf_{V\text{ satisfying some $L^1$ and $L^\infty$ constraints}} \lambda(V)=:\overline \lambda\end{equation} and establish a quantitative inequality. This was done, in the two-dimensional case and when the underlying domain is a ball, in \cite{MazariQuantitative}.

\paragraph{The turnpike-phenomenon}
The turnpike-phenomenon, on the other hand, deals with time-evolving control systems. Given a time horizon $T>0$ and a cost function $J_T=J_T(v)$ where $v=v(t,x)$ is the (time-evolving) control, such that $J_T$ has, under certain constraints, a minimiser $v^*$, then the turnpike principle states that $v^*$ should be close to the minimiser $u^*=u^*(x)$, under the same kind of constraints, of a stationary optimal control problem associated with a functional $\overline J$. Here closeness can be understood in different ways. The two most classical definitions are the exponential-turnpike \cite{TrelatZhang}, which gives a bound of the form
$$\forall t\in (0;T)\,, \Vert v^*(t,\cdot)-u^*(\cdot)\Vert_X\lesssim e^{-t}+e^{-(T-t)},$$ for a certain norm $X$, or the integral-turnpike \cite{LanceTrelatZuazua}, which is the version we will use in the main results and which gives estimates of the form
$$\int_0^T \Vert v^*(t,\cdot)-u^*(\cdot)\Vert_X^\alpha dt\leq M\text{ independent of }T,$$
for some norm $X$ and some exponent $\alpha>0$.  Over the recent years, this \emph{turnpike-phenomenon} has acquired a great importance in control theory, as exemplified by the numerous papers devoted to it. It has first been investigated in the context of econometry \cite{Dorfman}, but has since found applications in many fields and has been established in a variety of context. For a general introduction to this phenomenon, we refer to \cite{2006,Carlson1991} and we point, for recent results, to \cite{AftalionTrelat,LanceTrelatZuazua,PighinSakamoto,TrelatZhang,TrelatZhangZuazua,Zuazua2017}. We also mention \cite{borjan} for applications of the turnpike to machine learning. In the context of partial differential equations however, this property has only been investigated in the context of \emph{linear controls} \cite{LanceTrelatZuazua,Zuazua2017}, in which case the first step is to use the characterization of optimal controls, which are then finely analysed to prove that either at least one optimal control satisfying the turnpike property existst, or that all optimal controls satisfy it. It should also be noted that, apart from \cite{LanceTrelatZuazua}, most functionals for which the turnpike property is established contain a so-called tracking term, i.e. a term of the form
$\int_0^T \Vert y(t,\cdot)-\overline y(\cdot)\Vert_X^\alpha$ where $y$ is the state of the control problem and $\overline y$ is a stationary reference configuration.

In sharp contrast with these previous contributions, our work is, to the best of our knowledge, the first to address this question for \emph{bilinear optimal control problems}, and our methods do not rely on the characterization of optimal controls. Much rather, we will first develop specific tools to attack the question of stability estimates for a Schr\"{o}dinger operator and show how this directly implies an integral turnpike property, for an optimal control problem whose solution can not be characterized except in very particular geometries.

We believe that our techniques will serve to establish that quantitative estimates for optimal control problems, which are less developed than their optimal shapes counterparts, can be used to obtain qualitative information about the behaviour of solutions of intricate optimal control problems. 

\paragraph{Bilinear optimal control problems}

Although bilinear controllability is a very active field \cite{Alabau,Floridia}, the literature devoted to optimal bilinear control problems is rather scarce. We mention  \cite{Fister,GuillnGonzlez2020}, where bilinear optimal control problems for chemotaxis or chemorepulsion models are investigated in the context of systems. In these papers, the cost functional is of tracking type and the emphasis is put on existence properties as well as on derivation of optimality conditions. In \cite{Borzi}, an optimal control problem for bilinear controls is considered from the optimality conditions point of view. Their cost functional is of tracking type and special emphasis is put on the multigrid numerical analysis of the optimisation problem. Our paper is to the best of our knowledge the first contribution to the qualitative analysis of such problems. Explicit computations may allow one to obtain the optimal controls in very specific geometries \cite{Alvino1990} (see Remark \ref{Rem:ATL} below) but it is in general hopeless to get such characterization, so that proving turnpike for these problems provides a first valuable information.

In this article, the type of optimal control problems considered reads as follows: considering, in a domain $\O$, an initial datum $u_0\in L^\infty\,, \geq 0\,, u_0\neq 0$, a final time $T>0$ and a control potential $\mathcal V=\mathcal V(t,x)$ satisfying \emph{$L^1$ and $L^\infty$ constraints}, solve
\begin{equation}\label{Eq02}\tag{$\bold O_{II}$}\sup_{\mathcal V \text{ satisfying some $L^1$ and $L^\infty$ constraints}}\int_\O u_\mathcal V(T,\cdot)\,, \text{ subject to }
\begin{cases}\partial_t u_\mathcal V-\Delta u_\mathcal V-\mathcal Vu_\mathcal V=0\,, \text{ }t\geq 0\,, x \in \O\,, \\ u_\mathcal V(t,\cdot)=0\text{ on }\partial \O \,, \\ u_\mathcal V(0,\cdot)=u_0\text{ in }\O.
\end{cases}\end{equation}
 
We will prove that integral turnpike holds for \eqref{Eq02} and, more precisely, that the set of turnpike controls (i.e. the stationary controls the time-dependent optimal controls should stay close to) is the set of solutions of the spectral optimisation problem \eqref{Eq0} under the same $L^1$ and $L^\infty$ constraints.

\paragraph{Structure of the article}
\begin{itemize}
\item Subsection \ref{Su:Main} contains a presentation of both the spectral optimisation problem and of the bilinear optimal control problem as well as the statement of the main results. Theorem \ref{Th:Quanti} deals with the quantitative inequality, Theorem \ref{Th:Turnpike} with the turnpike phenomenon.
\item Section \ref{Se:Quanti}, which takes up most of the paper, is devoted to the proof of Theorem \ref{Th:Quanti}.
\item Section \ref{Se:Turnpike} is devoted to the proof of Theorem \ref{Th:Turnpike}.
\item In the Appendices, we gather some technical proofs.
\end{itemize}

\subsection{Presentation of the problems and main results}\label{Su:Main}

\subsubsection{The spectral optimisation problem \& the quantitative inequality}

\paragraph{The optimisation problem}

To state the optimisation problem and the quantitative inequality, let us consider a domain $\O$ with a $\mathscr C^3$ boundary.

Let us define, for any function $V \in L^\infty(\O)$, $\lambda(V)$ as the first eigenvalue of the operator 
$$ L_V:=-\Delta -V$$ with Dirichlet boundary conditions. 

One can define $\lambda(V)$ using the Rayleigh quotient formulation
\begin{equation}\label{Eq:Rayleigh}
\lambda(V):=\inf_{u \in W^{1,2}_0(\O)\,, u\neq 0}\frac{\displaystyle \int_\O |\n u|^2-\int_\O V u^2}{\displaystyle \int_\O  u^2}.\end{equation}

It is classical to see that this eigenvalue is simple and that any eigenfunction associated with this eigenvalue has constant sign. We hence define, for any $V\in L^\infty(\O)$,  $u_V$ as the unique eigenfunction associated with $\lambda(V)$ satisfying 
\begin{enumerate}
\item $u_V>0$ in $\O$, $u_V=0$ on $\partial \O$,
\item $\int_\O u_V^2=1.$
\end{enumerate}
In particular, $u_V$ satisfies 
\begin{equation}\label{Eq:Eig}
\begin{cases}
-\Delta u_V-Vu_V=\lambda(V)u_V\text{ in }\O\,,
\\ u_V=0\text{ on }\partial \O.\end{cases}\end{equation}
Let us fix a parameter $V_0\in (0;|\O|)$ and define the admissible class
\begin{equation}\label{Eq:Adm}\tag{$\bold{Adm}$}\mathcal M:=\left\{0\leq V \leq 1\,, \int_\O V=V_0\right\}.\end{equation} 
The optimisation problem is then 
\begin{equation}\label{Eq:E1}\tag{$\bold{P}_\lambda$}\fbox{$\displaystyle \overline \lambda:=\inf_{V\in \mathcal M}\lambda(V).$}\end{equation}

It is easy to show that this minimisation problem has at least a solution $V^*$ using the direct method of calculus of variations, see \cite{LamboleyLaurainNadinPrivat}. This optimisation problem is intimately linked to the question of survival of species in classical models of mathematical biology \cite{BHR,CantrellCosner1}. In that case, the potential $V$ model the distribution of resources accessible to a particular species, and the sign of $\lambda(V)$ governs the long-time behaviour of the classical logistic-diffusive model. In that \textquotedblleft resources distribution" interpretation, \eqref{Eq:E1} amounts to solving the following problem: what is the best way to spread resources in a domain to ensure survival of a species? For more references on the application of \eqref{Eq:E1} to mathematical biology problems, we refer to the introduction of \cite{MazariThese}. This problem also has many applications in the mathematical modelling of composite membrane \cite{Cox1991,Cox1990} and can exhibit a variety of behaviours depending on the parameters \cite{Chanillo2000}.

The following question, which deals with the stability of a minimiser $V^*$, is natural in shape optimisation:
\begin{equation}\tag{$\bold Q_0$}\label{Q0}
\text{\emph{Can we estimate the remainder $\lambda(V)-\lambda(V^*)$ from below using $\Vert V-V^*\Vert_{L^1}$?} }\end{equation}

Obviously, since uniqueness may not hold for \eqref{Eq:E1}, for this question to be relevant, one needs to define 
the set of minimisers
\begin{equation}\tag{$\bold{Opt}$}\mathcal I^*:=\left\{V^*\text{  solution of \eqref{Eq:E1}}\right\},\end{equation}
and  \eqref{Q0} is then reformulated as 

\begin{equation}\tag{$\bold Q_1$}\label{Q1}\text{
\emph{Can we estimate the remainder $\lambda(V)-\overline \lambda$ from below using $\operatorname{dist}_{L^1}\left(V,\mathcal I^*\right)$?}} \end{equation}

\paragraph{Assumptions on the domain $\O$}
To answer \eqref{Q1}, one needs two technical assumptions on the domain $\O$ itself. More specifically, it is possible to prove \cite{KaoLouYanagida,LamboleyLaurainNadinPrivat} that any solution $V^*$ of \eqref{Eq:E1} is a bang-bang function i.e. that there exists a measurable subset $E^*\subset \O$ such that
$$\operatorname{Vol}(E^*)=V_0\,, V^*=\mathds 1_{E^*}.$$ Any subset $E^*\subset \O$ such that $\mathds 1_{E^*}\in \mathcal M$ and $\mathds 1_{E^*}$ solves \eqref{Eq:E1} is called an \textit{optimal spectral set}. This leads to introducing the set of optimal spectral sets 
\begin{equation}\tag{$\bold{Opt_{s}}$}\mathcal E^*:=\left\{E\subset \O\,, \text{ $\mathds 1_E\in \mathcal M$ and $\mathds 1_E$ solves \eqref{Eq:E1}}\right\}.\end{equation} Since every solution of \eqref{Eq:E1} is a bang-bang function, there is a bijection between $\mathcal E^*$ and $\mathcal I^*$.
 
The proof of the quantitative inequality will involve at some point \textit{shape derivatives at optimal spectral sets}. In order to apply it one needs to be able to compute second order shape derivatives, and  we are thus led to assume the following:
\begin{gather*}
\label{A1}\tag{$\bold A_1$}
\text{ Any optimal spectral set $E^*$  of $\O$ has a $\mathscr C^3$ boundary.}
\\\text{ If $H_{E^*}$ is the mean curvature of $\partial E^*$, 
$H_\infty:=\sup_{E^*\in \mathcal E^*}\Vert H_{E^*}\Vert_\infty<+\infty.$}
\\\text{ Finally, if $\operatorname{Per}(E^*)$ is the Cacciopoli perimeter of $E^*$, $\operatorname{Per}_\infty:=\sup_{E^*\in \mathcal E^*} \operatorname{Per}({E^*})<+\infty.$}
\end{gather*}
This $\mathscr C^3$ assumption is standard in that context \cite{DambrineLamboley}.

\begin{remark}[Regarding the $\mathscr C^3$ assumption] \textit{It would be interesting to see whether or not it could be possible to lower these regularity assumptions and obtain quantitative inequalities using, for instance, recent structure results for shape derivatives at Lipschitz sets \cite{Laurain2020}. We however expect several technical difficulties.}
\end{remark}
\begin{remark}[Some domains $\O$ satisfying \eqref{A1}]
\emph{The main regularity results for \eqref{Eq:E1} can be found in \cite{Chanillo2008}, where it is proved that, when $d=2$ and $\O$ is Lipschitz regular, then the boundary of any optimal spectral set $E^*\in \mathcal E^*$ consists of finitely many disjoints analytic curves. In higher dimensions, the boundary of  any optimal spectral set is known to be smooth up to a closed set with Hausdorff dimension at most equal to $d-1$ \cite{Chanillo2008bis}. The main focus of this paper is however not the regularity properties of optimal spectral sets, and we thus choose to work under such "simplifying" assumptions.}
\end{remark}

Once \eqref{A1} is satisfied, it is possible to state our \emph{non-degeneracy assumption}. Indeed, as is customary in shape optimisation \cite{DambrineLamboley}, one always needs to assume some kind of coercivity on second order shape-derivatives. In order to state this assumption, let us briefly recall some definitions of shape derivatives:

\begin{definition}\label{De:ShapeStab}

Let $\mathcal F:E\mapsto \mathcal F(E)\in \R$ be a shape functional. We define, for $E^*\in \mathcal E^*$,
$$\mathcal X_1(E^*):=\left\{\Phi\in W^{3,\infty}_c(\O;\R^n),  \forall t \in(-1;1)\, , (\operatorname{Id}+t\Phi)(E^*)\subset \O.\right\}$$as the set of admissible perturbations at $E^*$.
The shape derivative of first (resp. second) order of a shape functional $\mathcal F$ at $E^*$ in the direction $\Phi$ is 
\begin{multline}
\mathcal F'(E^*)[\Phi]=\lim_{t\to 0}\frac{\mathcal F\Big((Id+t\Phi)E^*\Big)-\mathcal F(E^*)}t
\\\text{ (resp.}\mathcal F''(E^*)[\Phi,\Phi]:=\lim_{t^2\to 0}\frac{\mathcal F\Big((Id+t\Phi)E^*\Big)-\mathcal F(E^*)-\mathcal F'(E^*)(\Phi)}{t^2}.\text{)}\end{multline} provided it exists.

We say that $E^*$ satisfies first (resp. second) order  shape optimality condition for $\mathcal F$ if, for any $\Phi\in \mathcal X_1(E^*)$,  $\mathcal F'(E^*)[\Phi]$ is well defined and $\mathcal F'(E^*)[\Phi]=0$ (resp. both $\mathcal F'(E^*)[\Phi]$ and $\mathcal F''(E^*)[\Phi,\Phi]$ are well defined and $\mathcal F'(E^*)[\Phi]=0$, $\mathcal F''(E^*)[\Phi,\Phi]>0$). 
\end{definition}

Using that definition, we can now define a notion of shape optimality for \eqref{Eq:E1}. Namely, with a slight abuse of notation we can define, for any measurable subset $E\subset \O$, $$\lambda(E):=\lambda(\mathds 1_E)$$ and consider for any real parameter $\Lambda\in \R$, the Lagrangian
\begin{equation}\label{Eq:Lagrangian}\mathcal L_\Lambda(E):=\lambda(E)-\Lambda \operatorname{Vol}(E).\end{equation} This Lagrangian is used to handle the volume constraint. At any optimal spectral set $E^*\in \mathcal E^*$, the theory of Lagrange multipliers proves that there exists a Lagrange multiplier (that we can compute, see Remark \ref{Re:LagrangeMultiplier}) $\Lambda(E^*)$ such that, for any $\Phi \in \mathcal X_1^*(E^*)$, 
\begin{equation}\mathcal L_{\Lambda(E^*)}'(E^*)[\Phi]=0\,, \mathcal L_{\Lambda(E^*)}''(E^*)[\Phi,\Phi]\geq 0.\end{equation} For the sake of readability, we write 
$$L_{E^*}:=\mathcal L_{\Lambda(E^*)}.$$ The second Assumption we make is:
\begin{gather*}
\text{Any optimal spectral set $ E^*$ satisfies first order optimality conditions for the Lagrangian}\\\text{  $L_{E^*}$,  and  furthermore there exists   $\alpha>0$ such that for any optimal set $E^*$, }
\\\text{the quadratic form $\ell_{E^*}= L''_{E^*}$ satisfies}\\\text{ $ \ell_{E^*}(\Phi\cdot \nu,\Phi \cdot \nu )\geq \alpha||\Phi \cdot \nu||_{L^2(\partial E^*)}^2$} \text{ for any $\Phi \in \mathcal X_1(E^*)$}\label{A2}\tag{$\bold A_2$}
\end{gather*}

\begin{remark}[Comment on this coercivity assumption]\emph{
Two things should be said about \eqref{A2}. The first one is that, contrary to the traditional coercivity assumption in shape optimisation, which involves $H^{\frac12}$ norms, it has been proved in \cite{MazariQuantitative} that, when $\O$ is a ball, this $L^2$ coercivity norm is optimal. Furthermore, as will be proved later in this article, all the remainder terms can also be bounded using this $L^2$ norm, which seems to indicate that this is indeed the optimal coercivity norm. We do however note that proving this for general domains seems quite arduous given that, in \cite{MazariQuantitative} this coercivity is proved using explicit computations, diagonalization of the second order shape derivative of the Lagrangian and ad-hoc comparison principles.}

\emph{The second remark one can make is that the uniformity of the constant $\alpha>0$ is crucial, as will be clear in the final steps of the proof.}
\end{remark}


\paragraph{Statement of the result} We are now in a position to state our main result:

\begin{theorem}\label{Th:Quanti}
Assume $\O$ is a $\mathscr C^2$ bounded domain in $\R^d$. Assume that $\O$ satisfies Assumptions \eqref{A1} and \eqref{A2}. 
Then there exists $C=C(V_0,\O)$ such that 
\begin{equation}\label{Eq:Quanti}
\forall V \in \mathcal M\,, \lambda(V)-\lambda(V^*)\geq C\inf_{V^*\in \mathcal I^*}||V-V^*||_{L^1}^2= C \operatorname{dist}_{L^1}(V,\mathcal I^*)^2.\end{equation}
\end{theorem}

Such quantitative inequalities have been studied in great details using shape optimisation when the optimisation parameter is the set $\O$ itself rather than the potential $V$ \cite{BDPV,FuscoMaggiPratelli}. Let us briefly mention that this inequality was established in the case $\O=\mathbb B(0;1)$ in \cite{MazariQuantitative}, and we also mention \cite{BrascoButtazzo,CarlenLieb}, where quantitative inequalities for potential optimisation problems are obtained in other settings, for $L^p$  constraints.  In this article, due to the particular nature of our constraints, our methods are significantly different from \cite{BrascoButtazzo,CarlenLieb}. We do emphasise that we will in the proof give an alternative proof of a quantitative Hardy-Littlewood inequality (here stated in the form of \textquotedblleft bathtub principle"). This inequality, Proposition \ref{Pr:Bathtub},  is linked to the results \cite{Cianchi} but, in our context, the amount of information we have allows for a more straightforward proof suited to our needs. Let us also note that quantitative inequalities for the Riesz-Sobolev rearrangement inequalities, which are crucial in rearrangements, were obtained recently \cite{Christ,Lieb} For further references in spectral quantitative inequalities, we refer to the survey paper \cite{BDPV}.

\subsubsection{The optimal bilinear control problem \& the turnpike phenomenon}\label{Su:PT}
In the second part, we consider the following evolution equation 
\begin{equation}\label{Eq:Main}
\begin{cases}\partial_t y_\mathcal V-\Delta y_\mathcal V-\mathcal Vy_\mathcal V=0\,, \text{ }t\geq 0\,, x \in \O\,, \\ y_\mathcal V(t,\cdot)=0\text{ on }\partial \O \,, \\ u_\mathcal V(0,\cdot)=y_0\geq 0\,, y_0\in L^\infty(\O)\,, y_0\neq 0\text{ in }\O.
\end{cases}\,, \text{ under the constraint that for a.e $t\geq 0\,, \mathcal V(t,\cdot) \in \mathcal M$}.\end{equation}
The set  $\mathcal M$ was defined in \eqref{Eq:Adm}. We define 
\begin{equation}\tag{$\bold{Adm_{evol}}$}\label{Eq:AdmEvol}\mathcal M_{\R_+}:=\left\{\mathcal V\,, \text{for a.e $t\geq 0\,, \mathcal V(t,\cdot) \in \mathcal M$}\right\}.\end{equation}

This equation is a  crude model for linear growth; in that context, $\mathcal V(t,\cdot)$ represent the resources distribution at a time $t$.

Following that interpretation we want, for some fixed time horizon $T>0$,  to maximise the total population size with respect to the control $\mathcal V$, that is, we want to solve 
\begin{equation}\tag{$\bold{P}_T$}\label{Eq:PT}\fbox{$\displaystyle 
\sup_{\mathcal V\in \mathcal M_{\R_+}}\int_\O y_\mathcal V(T,x)dx.$}\end{equation}The existence of optimal controls is straightforward  and follows from the direct method in the calculus of variations. On the other hand, the qualitative properties of the optimal controls are much more complicated to analyse. 
\begin{remark}[Comment on the type of functionals]\emph{
It is notable that, in contrast to several works dealing with turnpike or bilinear optimal control problems \cite{GuillnGonzlez2020,LanceTrelatZuazua,PighinSakamoto}, we do not work with a global $L^2((0;T)\times \O)$ constraint on the control, and that the functional under consideration is not of tracking type, that is, it does not contain a term of the form 
$$\int_0^T \Vert y_{\mathcal V}-\overline y\Vert_{L^2(\O)}^2$$ for some reference configuration or trajectory $\overline u.$ For instance, in  \cite{Borzi}, the same type of state equation is considered but the functional involves a tracking term.}\end{remark}
Let us then fix an optimal control $\mathcal V_T^*$. Our main question here is 

\begin{equation}\tag{$\bold Q_2$}\label{Q2}\text{\emph{Is it true that, for most of the time, $\mathcal V_T^*$ is close to a static control $V\in \mathcal M$?}}\end{equation}
The main difficulty is that the bilinearity of the control makes all the existing methods used to prove the turnpike property inapplicable, as they all deal with linear controls \cite{LanceTrelatZuazua,Zuazua2017}. Furthermore, the $L^\infty$- constraint leads to potential difficulties, as it evades the classical $L^2$ setting of the turnpike property.

We prove that, when the time-horizon is large, this property holds for the set of optimal potential for the spectral minimisation problem \eqref{Eq:E1} using the quantitative inequality established in the first part. 

\begin{remark}[Solution of \eqref{Eq:PT} in radial geometry]\label{Rem:ATL}\emph{
In \cite{Alvino1990},  the problem \eqref{Eq:PT} is explicitly solved when $\O$ is a ball. It is shown that the unique solution of this optimal control problem is in fact a uniquely characterised static distribution $\mathcal V_T^*(t,x)\equiv \mathds 1_{\mathbb B^*}$, where $\operatorname{Vol}(\B^*)=V_0$. The proof relies on  involved Talenti-type inequalities and on very fine properties of the distribution functions. Such tools are not available in other geometries, and the turnpike property can thus be seen as a weaker generalization of these results to other geometries.
}\end{remark}

\begin{theorem}\label{Th:Turnpike}
Let $\O$ satisfy \eqref{A1}-\eqref{A2}. There exists $M>0$ such that, for any $T>0$ and any solution  $\mathcal V_T^*$ of \eqref{Eq:PT} there holds
\begin{equation}
\int_0^T \operatorname{dist}_{L^1}\left(\mathcal V_T^*(t,\cdot),\mathcal I^*\right)^2dt\leq M.
\end{equation}
\end{theorem}

\paragraph{Notational conventions}
\begin{itemize}
\item When $E\subset \O$ is a measurable subset, the notation $u_E$ stands for $u_{\mathds 1_E}$ defined in\eqref{Eq:Eig}.
\item When $E\subset \O$ is a measurable subset, the notation $\lambda(E)$ is a shorthand for $\lambda(\mathds 1_E)$.
\end{itemize}

\section{Proof of Theorem \ref{Th:Quanti}: A quantitative spectral inequality}\label{Se:Quanti}

\subsection{Structure of the proof}To prove Theorem \ref{Th:Quanti}, we proceed in several steps, as was done in \cite{MazariQuantitative}. Several complications arise here, in this more general context.

\begin{itemize}
\item \underline{Step 0:} As a preliminary step, we gather some technical information about the eigenfunctions and their normal derivatives.
\item  \underline{Step 1:} The first crucial step is to obtain local quantitative inequalities for normal deformations of optimal sets. 
This is done using Assumptions \eqref{A1}-\eqref{A2} and following the general strategy synthesised in \cite{DambrineLamboley}.
\item \underline{Step 2:} We introduce the following problem inspired by \cite{AcerbiFuscoMoreni}: for any parameter $\delta$, define the class 
$$\mathcal I(\delta):=\left\{V \in \mathcal M\,, \operatorname{dist}_{L^1}(V,\mathcal I^*)=\delta\right\}$$ as well as the auxiliary problem 
\begin{equation}\label{Eq:09}
\inf_{V\in \mathcal I(\delta)}\big(\lambda(V)-\overline\lambda\big).\end{equation}
We will prove (Proposition \ref{Pr:Exist}) that this optimisation problem has a solution $V_\delta$, write down  the optimality conditions  and show that proving Theorem \ref{Th:Quanti} is equivalent to proving that 
$$\underset{\delta \to 0}{\lim\inf}\left(\frac{\lambda(V_\delta)-\overline \lambda}{\delta^2}\right)>0.$$ 
\item \underline{Step 3:} In this step, we  establish a quantitative bathtub principle for functions with regular level sets see Proposition \ref{Pr:Bathtub}.
\item \underline{Step 4:} Using all the previous steps as well as Assumptions \eqref{A1}-\eqref{A2}, we will give the proof of Theorem 1.  
\end{itemize} 
\paragraph{Heuristics for the proof} Since this proof is long, let us explain the idea graphically. Let us assume that we have a unique optimal spectral set $E^*$.

\begin{figure}[H]
\begin{center}
\begin{tikzpicture}
\draw[blue,thick,fill=blue, fill opacity=0.1] plot [smooth, tension=1.2] coordinates{ (2,0.1) (1.5,1) (1,1.3) (0.5, 1.6) (-0.1,1.3) (-1,0.9) (-1,0) (-0.2,-0.4) (1,-0.7) (2.1,-0.9) (2,0.1)};
\draw(0.5,0.5)node[scale=1] {$E^*$};
\end{tikzpicture}
\end{center}
\caption{Depiction of the optimal set $E^*$. }

\end{figure}
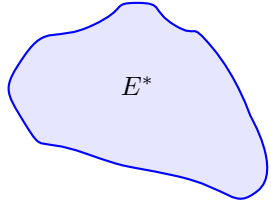

 We will argue by contradiction and assume that the quantitative inequality does not hold. For $\delta>0$ small enough, the solution $V_\delta$ of the auxiliary problem \eqref{Eq:09}, which we can expect to be the characteristic function of a set $E_\delta$, will be very close to $E^*$:

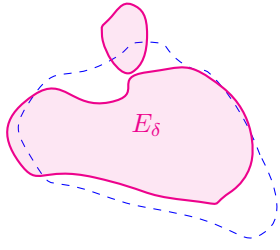
\begin{figure}[H]
\begin{center}
\begin{tikzpicture}
\draw[blue,thin,dashed, fill opacity=0.1] plot [smooth, tension=1.2] coordinates{ (2,0.1) (1.5,1) (1,1.3) (0.5, 1.6) (-0.1,1.3) (-1,0.9) (-1,0) (-0.2,-0.4) (1,-0.7) (2.1,-0.9) (2,0.1)};
\draw[magenta,thick, fill=magenta,fill opacity=0.1] plot [smooth, tension=1.2] coordinates{ (1.9,0.1) (1.5,1) (0.5, 1.2) (0,0.8) (-1,0.9) (-1.2,0) (-0.2,-0.2) (1,-0.5) (1.6,-0.4) (1.9,0.1)};
\draw[magenta](0.5,0.5)node[scale=1] {$E_\delta$};

\draw[magenta,thick, fill=magenta,fill opacity=0.1] plot [smooth, tension=1] coordinates{ (0,1.4) (-0.1,1.7) (-0.01,1.9) (0.3,2.1) (0.5,1.7) (0.3,1.2) (0,1.4)};

\end{tikzpicture}
\end{center}
\caption{Depiction of the solution $E_\delta$ of the auxiliary problem \eqref{Eq:09}; $E^*$ is depicted in dashed blue. }

\end{figure}

As a first step, we use the bathtub principle. Let $\mu_\delta>0$ be such that the level set of the eigenfunction $u_\delta$ associated with $\mathds 1_{E_\delta}$ satisfies $|\{u_\delta>\mu_\delta\}|=V_0$. We define $\tilde E_\delta:=\{u_\delta>\mu_\delta\}$:

\begin{figure}[H]
\begin{center}
\begin{tikzpicture}

\draw[magenta,thin, dashed] plot [smooth, tension=1.2] coordinates{ (1.9,0.1) (1.5,1) (0.5, 1.2) (0,0.8) (-1,0.9) (-1.2,0) (-0.2,-0.2) (1,-0.5) (1.6,-0.4) (1.9,0.1)};

\draw[magenta,thin,dashed] plot [smooth, tension=1] coordinates{ (0,1.4) (-0.1,1.7) (-0.01,1.9) (0.3,2.1) (0.5,1.7) (0.3,1.2) (0,1.4)};
\draw[orange](0.5,0.5)node[scale=1] {$\tilde E_\delta$};

\draw[orange,thick,fill=orange, fill opacity=0.1] plot [smooth, tension=1.3] coordinates{(2,-0.5) (2,0) (1.5,0.8) (1,1.4) (0.5, 2) (0,1.5) (-0.5,1) (-1,0) (-0.5,-1) (1,-0.7) (2,-0.5)};

\end{tikzpicture}
\end{center}\caption{Level set $\tilde E_\delta$ of the eigenfunction $u_\delta$ associated with $E_\delta$ and satisfying the volume constraint; $E_\delta$ is depicted in dashed orange. }

\end{figure}
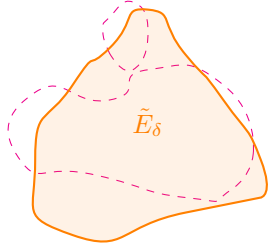

  Using  the Rayleigh quotient formulation \eqref{Eq:Rayleigh} of the eigenvalues,  we have $\lambda(E_\delta)\geq \lambda(\tilde E_\delta)$, and the quantitative bathtub principle (Proposition \ref{Pr:Bathtub}) ensures that $\lambda(E_\delta)\geq \lambda(\tilde E_\delta)+c\Vert \mathds 1_{E_\delta}-\mathds 1_{\tilde E_\delta}\Vert_{L^1}^2$ for some $c>0$.

However, when $\delta>0$ is small enough, $\tilde E_\delta$ is a sufficiently regular normal deformation of $E^*$:

\begin{figure}[H]
\begin{center}
\begin{tikzpicture}
\draw[blue,thin,dashed] plot [smooth, tension=1.2] coordinates{ (2,0.1) (1.5,1) (1,1.3) (0.5, 1.6) (-0.1,1.3) (-1,0.9) (-1,0) (-0.2,-0.4) (1,-0.7) (2.1,-0.9) (2,0.1)};

\draw[orange,thick] plot [smooth, tension=1.3] coordinates{(2,-0.5) (2,0) (1.5,0.8) (1,1.4) (0.5, 2) (0,1.5) (-0.5,1) (-1,0) (-0.5,-1) (1,-0.7) (2,-0.5)};

\end{tikzpicture}
\end{center}
\caption{Comparison of $\tilde E_\delta$ (in orange) and of $E_\delta$ (in dashed blue). }
\end{figure}
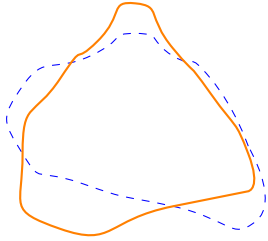
Using Step 2 and the quantitative inequality for local deformations, we will then get $\lambda(\tilde E_\delta)\geq \lambda(E^*)+c'\Vert \mathds 1_{\tilde E_\delta}-\mathds 1_{E^*}\Vert_{L^1}^2$ for some $c'>0$. Combining these two steps gives 
$$\lambda(E_\delta)\geq \lambda(E^*)+c'' \left(\Vert \mathds 1_{\tilde E_\delta}-\mathds 1_{E^*}\Vert_{L^1}^2+\Vert \mathds 1_{\tilde E_\delta}-\mathds 1_{E_\delta}\Vert_{L^1}^2\right)$$ for some $c''>0$. Since 
$\Vert \mathds 1_{ E_\delta}-\mathds 1_{E^*}\Vert_{L^1}=\delta$ either $\Vert \mathds 1_{\tilde E_\delta}-\mathds 1_{E^*}\Vert_{L^1}$ or $\Vert \mathds 1_{\tilde E_\delta}-\mathds 1_{E_\delta}\Vert_{L^1}$ should be, as $\delta \to 0$, of order $\delta$, which then gives the required contradiction.      

To work alongside these steps, we however need several basic regularity results on the eigenfunctions $u_V$ defined in Equation \eqref{Eq:Eig}, and some properties on the optimisation problem \eqref{Eq:E1} itself. We include these informations here as a preliminary step.

\subsection{Step 0: Technical preliminaries}

We begin with basic regularity estimates on eigenfunctions:
\begin{lemma}\label{Le:Regularity}
For any $s\in [0;1]$, there exists $\mathscr Y(s)>0$ such that 
$$\forall V\in \mathcal M\,, \Vert u_V\Vert_{\mathscr C^{1,s}}\leq \mathscr Y(s).$$ For any $p\in (1;+\infty)$, there exists $\mathscr W(p)$ such that 
$$\forall V\in \mathcal M\,, \Vert u_V\Vert_{W^{2,p}}\leq \mathscr W(p).$$ \end{lemma} 

This Lemma relies on a classical bootstrap method; we postpone its proof to Appendix \ref{An:Technical}.

We will make a repeated use of the following semi-continuity property. Although it is a classical result, we state it as a Lemma and prove it in Appendix \ref{An:Technical}:
\begin{lemma}\label{Le:LSCeigenvalue}
Let $\{V_k\}_{k\in \N}\in \mathcal M^\N$. Assume that there exists $V_\infty \in \mathcal M$ such that $\{V_k\}_{k\in \N}$ converges  weak $L^\infty-*$ to $V_\infty$. Then 
\begin{equation}\underset{k\to \infty}{\underline \lim}\lambda(V_k)\geq \lambda(V_\infty).\end{equation} 
\end{lemma}

A key component of the proof of the quantitative inequality is the non-degeneracy of eigenfunctions. Here, by non-degeneracy, we mean that the normal derivative at any point on the boundary is negative, and as a straightforward consequence of our Assumption \eqref{A1} we obtain the following result:

\begin{lemma}\label{Le:Hopf} Assume $\O$ satisfies Assumption \eqref{A1}. Then there exists a constant $c_0$ such that
$$\forall E^*\in \mathcal E^*\, , \min\left(u_{E^*}|_{\partial E^*}\,, \inf_{\partial E^*}\left(-\frac{\partial u_{E^*}}{\partial \nu}\right)_{|\partial E^*}\right)\geq c_0>0$$\end{lemma}

\begin{proof}[Proof of Lemma \ref{Le:Hopf}]
We first prove that for every $E^*\in \mathcal I^*$ there exists a constant $c(E^*)$ such that $$\min\left(u_{E^*}|_{\partial E^*}\,, \inf_{\partial E^*}\left(-\frac{\partial u_{E^*}}{\partial \nu}\right)_{|\partial E^*}\right)\geq c(E^*).$$

We first observe that, as the first eigenfunction $u_{E^*}$ was chosen to be non-negative, it follows from the maximum principle that 
$$u_{E^*}>0\text{ in }\O.$$

Furthermore, the optimality conditions for the spectral optimisation problem \eqref{Eq:E1} read \cite{KaoLouYanagida,LamboleyLaurainNadinPrivat,MazariQuantitative}: \begin{equation}\label{Eq:OptCon} \text{ There exists $\mu(E^*)$ such that }
E^*=\left\{u_{E^*}> \mu(E^*)\right\}.\end{equation} We comment on this in Remark \ref{ReBa}.

As a consequence of the fact that $V_0<|\O|$, we have 
$$u|_{\partial E^*}>0.$$

The regularity Assumption \eqref{A1} allows us to apply Hopf's Lemma to conclude that, for any $y\in \partial E^*$, there holds 
$$-\frac{\partial u_{E^*}}{\partial \nu}(y)>0.$$ To prove that there is a uniform lower bound, we then choose a minimising sequence $\{y_k\}_{k\in \N}\in (\partial E^*)^\N$. Since $\O$ is compact and $E^*\cap \partial \O=\emptyset$, there exists a closure point $y_\infty\in \O$ of this sequence. Passing to the limit in the equation
$$u_{E^*}(y_k)=\mu(E^*)$$ implies $u_{E^*}(y_\infty)=\mu(E^*)$ and thus $y_\infty\in \partial E^*$. The contradiction follows from the fact that, $u_{E^*}$ being $\mathscr C^{1,s}$ for any $s\in (0;1)$, 
$$\lim_{k\to \infty}\left(-\frac{\partial u_{E^*}}{\partial \nu}(y_k)\right)=\frac{-\partial u_{E^*}}{\partial \nu}(y_\infty)>0.$$

In order to make this lower bound uniform in $E^*\in \mathcal E^*$, let us define, for any $E^*\in \mathcal E^*$, $c(E^*)$ as the optimal constant in the previous step, that is 
$$c(E^*):=\min\left(u_{E^*}|_{E^*},\inf_{\partial E^*}\left(-\frac{\partial u_{E^*}}{\partial \nu}\right)\right),$$ and consider a minimising sequence $\{E_k^*\}_{k\in \N}\in \left(\mathcal E^*\right)^\N$ for $c(\cdot)$.

Let us consider a weak $L^\infty$-* limit $V_\infty^*\in \mathcal M$ of the sequence $\{V_k^*:=\mathds 1_{E_k^*}\}_{k\in \N}$. From Lemma \ref{Le:LSCeigenvalue} it follows that 
$$\lambda(V_\infty^*)\leq\lim_{k\to \infty}\lambda(V_k^*)=\inf_{V\in \mathcal M}\lambda(V)$$ so that $V_\infty^*$ is a solution of the spectral optimisation problem \eqref{Eq:E1}. We can thus write $V_\infty^*=\mathds 1_{E_\infty^*}$ for some $E_\infty^*\in \mathcal I^*$.  Furthermore, since the family $\{u_{E_k^*}\}_{k\in \N}$ is uniformly bounded in $\mathscr C^{1,s}$ for any $0<s<1$, the sequence $\{u_{E_k^*}\}_{k\in \N}$ converges strongly in $\mathscr C^1$ to $u_{E_\infty^*}$, so that 
$$\lim_{k\to \infty} c(E_k^*)=c(E_\infty^*)$$ up to a subsequence. Since $c(E_\infty^*)>0$, the conclusion follows.

\end{proof}

\begin{remark}\label{ReBa}[Comment on \eqref{Eq:OptCon}]\emph{Differentiating the map $m\mapsto \lambda(m)$ in the direction of a perturbation $h$ automatically yields the existence of a Lagrange multiplier $\mu(E^*)$ such that 
$$\{ u_{E^*}>\mu(E^*)\}\subset E^*\subset \{u_{E^*}\geq \mu(E^*)\}.$$ To derive \eqref{Eq:OptCon} (with the strict inequality sign) the easiest way to proceed is to observe, as is done in \cite{Chanillo2008,LamboleyLaurainNadinPrivat}, that \eqref{Eq:E1} is equivalent to minimising the first eigenvalue $$\tau(m)=\inf_{u\in W^{1,2}_0(\O)\,, \int_\O mu^2 \neq 0}\frac{\int_\O |\n u|^2}{\int_\O mu^2}$$ under the constraint $m\in \mathcal M$. In that case, the underlying eigenvalue equation is $-\Delta v=\tau(m)mv$, and it is then easier to see on this formulation that the level curve has Lebesgue measure zero.}\end{remark}


\subsection{First step: Quantitative inequality for normal deformation of optimal sets}


The first step is to obtain a local quantitative inequality at $E^*$, which, contrary to the results obtained in \cite{MazariQuantitative}, will be stated using the $W^{2,p}(\O,\R^d)$ norms of the perturbations. This setting is inspired by the synthetic presentation \cite{DambrineLamboley}, and is susceptible of being generalised to other local quantitative inequalities for optimal control problems. In that setting, we consider admissible vector fields, that is, vector fields $\Phi$ compactly supported in $\O$, and the deformed set 
$$E^*_{t\Phi}:=(Id+t\Phi)E^*.$$
We assume that $\Phi\in W^{2,p}(\O;\R^d)$, where $p$ is large enough so that the Sobolev embedding 
$$W^{2,p}(\Omega;\R^d)\hookrightarrow W^{1,\infty}(\O;\R^d)$$ holds. For instance, choosing $p>\frac{d}2$ ensures this.  We will however need higher regularity, as will be detailed through the next results.

\begin{proposition}\label{Pr:LocalStabilityUn}
Under Assumptions \eqref{A1}-\eqref{A2}, there exists $p\in (1;+\infty)$ as well as two constants $\eta\,,  \underline\alpha>0$ such that, for any $E^*\in \mathcal E^*$, for any $\Phi\in \mathcal X_1(E^*)\cap W^{2,p}(\O;\R^d)$ satisfying 
$$\Vert \Phi\Vert_{W^{2,p}(\O;\R^d)}\leq \eta$$ there holds 
\begin{equation}\tag{$\bold{I}_{local,E^*}$}\label{Eq:LocalStabilityUn}L_{E^*}(E^*_\Phi)-L_{E^*}(E^*)\geq \underline\alpha\left|E^*_\Phi\Delta E^*\right|^2\end{equation} where $L_{E^*}$ is the Lagrangian defined in \eqref{Eq:Lagrangian}.
If in particular $\operatorname{Vol}(E_\Phi^*)=\operatorname{Vol}(E)$ then \eqref{Eq:LocalStabilityUn} rewrites
$$\lambda(E_\Phi^*)-\lambda(E^*)\geq \underline\alpha\left|E^*_\Phi\Delta E^*\right|^2.$$
\end{proposition}

The proof of Proposition \ref{Pr:LocalStabilityUn} relies on several fine properties of first and second order shape derivatives. 

\subsection{Strategy of proof}


The idea of the proof follows the systematic presentation of \cite{DambrineLamboley}. Let $E^*\in \mathcal E^*$. For any vector field $\Phi\in \mathcal X_1(E^*)\cap W^{2,p}(\O;\R^d)$, define the function 
$$j_\Phi:[0;1]\ni t\mapsto L_{E^*} (E^*_{t\Phi}).$$
The fact that $E^*$ is a critical shape for $\lambda$ implies that, for any $\Phi\in \mathcal X_1(E^*)\cap W^{2,p}(\O;\R^d)$, 
\begin{equation}
j_{\Phi}'(0)=L_{E^*}(E^*)[\Phi]=0.   
\end{equation}
We can thus write 
\begin{align}
L_{E^*}(E^*_\Phi)-L_{E^*}(E^*)&=j_\Phi(1)-j_{\Phi}(0)
\\&=j_\Phi'(0)+\int_0^1 j_\Phi''(\xi)d\xi
\\&=j_\Phi''(0)+\int_0^1 \left(j_\Phi''(\xi)-j_\Phi''(0)\right)d\xi
\end{align}
From Assumption \eqref{A2} we have
$$j_\Phi''(0)\geq \mathscr \alpha\Vert \Phi \cdot \nu \Vert_{L^2(\partial E^*)}^2.$$
The crucial point is then to prove the following: there exists $p\in (1;+\infty)$, a constant $\mathscr M>0$ and a uniform modulus of continuity $\omega=\omega( \Vert \Phi\Vert_{W^{2,p}})$ that is, a function $\omega$ such that 
$$\forall \delta>0\,,  \exists \eta>0\,, \forall E^*\in \mathcal E^*\,, \forall \Phi\in \mathcal X_1(E^*)\cap W^{2,p}(\O;\R^d)\,, \Vert \Phi\Vert_{W^{2,p}}\leq \eta\Rightarrow 0\leq \omega( \Vert \Phi\Vert_{W^{2,p}})\leq \delta,$$

such that, whenever $ \Vert \Phi\Vert_{W^{2,p}}$ is small enough, for any $\xi \in (0;1)$,
\begin{equation}\label{Eq1}
\vert j_\Phi''(\xi)-j_\Phi''(0)\vert \leq \mathscr M \omega( \Vert \Phi\Vert_{W^{2,p}}) \Vert \Phi \cdot \nu \Vert_{L^2(\partial E^*)}^2.
\end{equation}

\begin{lemma}\label{Le:F}
Under Assumption \eqref{A1}, \eqref{Eq1} implies Proposition \ref{Pr:LocalStabilityUn}.
\end{lemma}
\begin{proof}[Proof of Lemma \ref{Le:F}]

If \eqref{Eq1} holds then there exists a constant $\alpha_1>0$ such that, whenever $E^*\in \mathcal E^*$ and whenever $ \Vert \Phi\Vert_{W^{2,p}}$ is small enough we have, by the Cauchy-Schwarz inequality 
$$L_{E^*}(E^*_\Phi)-L_{E^*}(E^*)\geq \alpha_1\int_{\partial E^*}(\Phi \cdot \nu)^2\geq \frac{\alpha_1}{\operatorname{Per}(E^*)^2}\left(\int_{\partial E^*}|\Phi \cdot \nu|\right)^2=\alpha(E^*)|E_\Phi^*\Delta E^*|^2$$ with $\alpha(E^*)= \frac{\alpha_1}{\operatorname{Per}(E^*)^2}$. From Assumption \eqref{A1} there exists a uniform constant $\underline \alpha$  such that, for $p$ large enough, for any $E^*\in \mathcal E^*$ and whenever $ \Vert \Phi\Vert_{W^{2,p}}$ is small enough we have, by the Cauchy-Schwarz inequality 
$$L_{E^*}(E^*_\Phi)-L_{E^*}(E^*)\geq \underline\alpha |E_\Phi^*\Delta E^*|^2.$$ If we now assume that 
$$|E_\Phi^*|=|E^*|=V_0$$ we obtain 
$$\lambda(E_\Phi^*)-\lambda(E^*)\geq \underline\alpha |E_\Phi^*\Delta E^*|^2$$ under the same assumptions on $\Phi$.  
\end{proof}

The rest of this Section is devoted to the proof of \eqref{Eq1}, that is, to the proof of the existence of a uniform modulus of continuity. This is done in  Proposition \ref{Pr:Continuity} below. Since this proof is technical, we first lay out the main useful results.

\subsection{Technical material for Proposition \ref{Pr:LocalStabilityUn}}

\paragraph{First and second order shape derivatives for $L_{E^*}$}

We recall, without proof, the following expressions for the first order shape derivative of both the Lagrangian and of the eigenfunctions. First of all, from \cite{MazariQuantitative}, one has that the shape derivative $u'_{E,\Phi}$ satisfies
\begin{equation}\label{Eq:DeriveeUnU}\begin{cases}
-\Delta u'_{E,\Phi}=\lambda(E)u'_{E,\Phi}+\mathds 1_E u'_{E,\Phi}+\lambda'(E)[\Phi]u_E\text{ in }\Omega\,, 
\\
u'_{E,\Phi}=0\text{ on }\partial \Omega\,, 
\\ \left[\frac{u'_{E,\Phi}}{\partial \nu}\right]=-\left(\Phi\cdot\nu\right) u_E \text{ on }\partial E.\end{cases}\end{equation}  

As a a consequence, we obtain the following expression for the first order derivative \cite{MazariQuantitative}:
\begin{equation}\tag{$\bold{DL_{E^*}}$}\label{Eq:DeriveeUnL}
L'_{E^*}(E)[\Phi]=-\int_{\partial E} u_E^2\left(\Phi\cdot\nu\right)+u_{E^*}^2|_{\partial E^*}\int_{\partial E}\left(\Phi\cdot\nu\right).
\end{equation}
\begin{remark}[Computation of the Lagrange multiplier with the first order shape derivative]\label{Re:LagrangeMultiplier} \textit{We note that this method can be used to find the value of the Lagrange multiplier associated with the volume constraint. Indeed, if we were looking for the Lagrange multiplier $\Lambda(E^*)$ at $E^*$ associated with the volume constraint one would use that for any $\Phi\in W^{1,\infty}(\O;\R^d)$ one should have 
$$-\int_{\partial E^*} (u_E^*)^2\left(\Phi\cdot\nu\right)-\Lambda(E^*)\int_{\partial E^*}\left(\Phi\cdot\nu\right)=0$$ so that $u_{E^*}$ is constant on $\partial E^*$ and $\Lambda(E^*)=-u_{E^*}^2|_{\partial E^*}.$}\end{remark}

 $E^*$ being a critical shape for the Lagrangian $L_{E^*}$  we can henceforth, using \cite[Theorem 5.9.2]{HenrotPierre}, work under the assumption\begin{equation}\tag{$\bold T_0$} \text{$\Phi$ is orthogonal to $\partial E^*$ .}\end{equation} We will now write down the second order shape derivative for $E^*_{t\Phi}$ for any $t\in[0;1]$.  To write down this second order shape derivative of the Lagrangian in a tractable way, let us introduce 
\begin{align*}T_t:=Id+t\Phi\,, J_{\Sigma,t}(\Phi):=\det(\n T_t)\left| (^T\n T_t^{-1})\nu\right|\,,\\ J_{\O,t}(\Phi):=\det(\n t\Phi)\,, A_t:=J_{\O,t}(\Phi)(Id+t\n \Phi)^{-1}(Id+t^T\n \Phi)^{-1},\end{align*} and  $\hat u_t:=u_{E^*_{t\Phi}}\circ T_t$.
Then the function $\hat u'_{E^*_{t\Phi},\Phi}:=u'_{E,t\Phi}\circ T_t$ which we abbreviate as $\hat u'_t$, satisfies
\begin{equation}\begin{cases}-\nabla \cdot\left(A_t\n \hat u_t'\right)=J_{\O,t}(\lambda(E^*_{t\Phi})\hat u'_t+\mathds 1_{E^*}\hat u'_t+\lambda'(E^*_{t\Phi})\hat u_t )\text{ in }\O\,,\\ \left[A_t \frac{\partial \hat u_t'}{\partial \nu}\right]=-J_{\Sigma,t}\left(\Phi\cdot\nu\right) \hat u_t\,, 
\\ \hat u_t'=0\text{ on }\partial \O.\end{cases}\end{equation}

Furthermore, \cite[Lemma 6 and Equation 73]{MazariQuantitative} (a typo is present in Equation (73) of \cite{MazariQuantitative} given that a parenthesis is missing) gives the following expression for the second order shape-derivative of the Lagrangian:
\begin{multline}\tag{$\bold{D^2L_{E^*}}$}\label{Eq:DeriveeDeuxL}
L_{E^*}''(E)[\Phi,\Phi]=-2\int_{\partial E^*}J_{\Sigma,t}\hat u'_t\hat u_t\left(\Phi\cdot\nu\right)\\+\int_{\partial E^*}J_{\Sigma,t} \left\{\hat H_t (\hat u_t^2-\left.u_{E^*}^2\right|_{\partial E^*})-2\hat u_t\frac{\partial \hat u_t}{\partial \nu}\right\}\left(\Phi\cdot\nu\right)^2.
\end{multline}
where $H_t$ is the mean curvature of $E^*_{t\Phi}$ and $\hat H_t:=H_t\circ T_t.$

As a consequence, the remainder that needs to be estimated is 
\begin{align}
L_{E^*}''(E)[\Phi,\Phi]-L_{E^*}''[E^*][\Phi,\Phi]=&-2\int_{\partial E^*}\left\{J_{\Sigma,t}\hat u'_t\hat u_t-u_0'u_0\right\}\left(\Phi\cdot\nu\right)\tag{$\bold{R_0}(t,\Phi)$}\label{R0}
\\&-2\int_{\partial E^*}\left\{J_{\Sigma,t}\hat u_t\frac{\partial \hat u_t}{\partial \nu}-u_0\frac{\partial u_0}{\partial \nu}\right\}\left(\Phi\cdot\nu\right)^2\tag{$\bold{R_1}(t,\Phi)$}\label{R1}
\\&+\int_{\partial E^*}J_{\Sigma,t} \left\{\hat H_t (\hat u_t^2-\left.u_{0}^2\right|_{\partial E^*})\right\}\left(\Phi\cdot\nu\right)^2\tag{$\bold{R_2}(t,\Phi)$}\label{R2}
\end{align}
where, with a slight abuse of notation, we write $u_0=u_{E^*}$, $u_0'=u_{E^*}'[\Phi]$.

We now estimate the three terms \eqref{R0}-\eqref{R1}-\eqref{R2} separately, using a series of geometric, continuity and trace estimates.

\subsubsection{Geometric, continuity and trace estimates}
\paragraph{Geometric estimates}

We first gather some geometric estimates:

\begin{proposition}[Geometric estimates,{ \cite[Lemma 4.7]{DambrineLamboley}}]
For any $p\in (1;+\infty)$, for any $\Phi\in \mathcal X_1(E^*)\cap W^{2,p}(\O;\R^d)\cap W^{1,\infty}(\O;\R^d)$, there exists a constant $M_p$ independent of $E^*$ such that, for any $t\in (0;1)$:
\begin{itemize}
\item \begin{equation}\label{Eq:Jac}\Vert \hat J_{\Sigma,t}-1\Vert_{L^\infty(\partial E^*)}\leq M_p\Vert \Phi\cdot\nu\Vert_{W^{1,\infty}(\partial E^*)}.\end{equation}
\item \begin{equation}\label{Eq:Curv}\Vert \hat H_t-H_{E^*}\Vert_{L^p(\partial E^*)}\leq M_p \Vert \Phi \cdot\nu\Vert_{W^{2,p}(\partial E^*)}.
\end{equation}
\end{itemize}
\end{proposition}
\paragraph{Continuity estimates}

The first continuity estimate is a basic application of elliptic regularity:
\begin{proposition}[{\cite[Claim 4]{MazariQuantitative}}]\label{Pr:C1}
For every $\delta>0$ there exists $\eta>0$ such that, for any $\Phi$ satisfying $\Vert \Phi\Vert_{W^{2,p}}\leq \eta$, for any $t\in (0;1)$, there holds
\begin{equation}
\Vert \hat u_t-u_0\Vert_{\mathscr C^1}\leq \delta.
\end{equation}
\end{proposition}
\begin{proposition}[{\cite[Lemma 4.8, Lemma 4.9]{DambrineLamboley}}]\label{PR2}
Let $p>d$. There exists a constant $\tilde M_p>0$ such that, for any $t\in (0;1)$,
\begin{equation}
\Vert \hat u_t-u_0\Vert_{W^{1,p}(\O)}\leq \tilde M_p \Vert \Phi\Vert_{W^{2,p}(\O)},
\end{equation}
and there also exists $\eta_0>0$ such that, if $\Vert \Phi\Vert_{W^{2,p}(\O)}\leq \eta_0$, then for any $t\in (0;1)$
\begin{equation}\Vert \hat u_t'-u_0'\Vert_{W^{1,2}(\O)}\leq \tilde M_p \Vert \Phi\cdot\nu\Vert_{W^{\frac12,2}(\partial E^*)} \Vert \Phi\cdot\nu\Vert_{W^{2,p}(\partial E^*)}.\end{equation}
All these constants are independent of $E^*\in \mathcal E^*$.
\end{proposition}

This proposition is proved by direct adaptation of \cite[Lemma 4.8, Lemma 4.9]{DambrineLamboley}. We continue with a stronger estimate:
\begin{proposition}[Derivative estimate, {\cite[Estimates (75)-(76)]{MazariQuantitative}}]\label{PR3}
There exists a constant $M_p'>0$ and $\eta_0'$ such that, if $\Vert \Phi\Vert_{W^{2,p}(\O;\R^d)}\leq \eta_0'$ then for any $t\in (0;1)$, 
\begin{equation}
\Vert \hat u_t'\Vert_{W^{1,2}_0(\O)}\leq M_p'\Vert \Phi\cdot\nu\Vert_{L^2(\partial E^*)}.
\end{equation}
Furthermore, for every $\delta_0>0$, there exists $\eta_0''>0$ such that, if $\Vert \Phi\Vert_{\mathscr C^1(\O)}\leq \eta_0''$, then for any $t\in (0;1)$
\begin{equation}
\Vert \hat u_t'-u_0'\Vert_{W^{1,2}_0(\O)}\leq \delta_0\Vert \Phi\cdot\nu\Vert_{L^2(\partial E^*)},
\end{equation}
All these constants can be chosen independent of $E^*$.
\end{proposition}
This estimate is, in \cite{MazariQuantitative}, obtained via a bootstrap method. From Propositions \ref{PR2} and \ref{PR3}, we obtain 
\begin{proposition}\label{Pr:Continuity}
For any $p>d$ there exists $\eta_p>0$ and $\mathscr M_p>0$ such that, if $\Vert \Phi\Vert_{W^{2,p}(\O;\R^d)}\leq\eta_p$, then for any $t\in (0;1)$,
\begin{equation}
\Vert \hat u_t-u_0\Vert_{W^{1,2}(\O)}\leq \mathscr  M_p \Vert \Phi \cdot \nu \Vert_{L^2(\partial E^*)}.
\end{equation}
\end{proposition}

\paragraph{Trace estimate}

We now use trace estimates. The embedding 
$$W^{1,2}_0(\O)\hookrightarrow L^q(\partial E^*) $$ is bounded for any $q\in \left(1;2\frac{d-1}{d-2}\right)$ (with the convention that when $d=2$ the upper bound of the interval is $+\infty$), and even, as a consequence of \eqref{A1}, uniformly bounded in $E^*\in \mathcal E^*$. As a consequence, we have the following:
\begin{lemma}\label{Le:Crucial} Let $p>d$ and $\eta_p$ be given by Proposition \ref{Pr:Continuity}.  Then, for any $q\in \left(1;2\frac{d-1}{d-2}\right)$  there exists $\mathscr N_q>0$ such that for any $\Phi\in \mathcal X_1(E^*)\cap W^{1,\infty}\cap W^{2,p}(\O;\R^d)$  satisfying $\Vert \Phi\Vert_{W^{2,p}(\O;\R^d)}\leq \eta_p$ and for any $t\in (0;1)$,
$$\Vert \hat u_t-u_0\Vert_{L^q(\partial E^*)}\leq \mathscr N_q \Vert \Phi\cdot \nu \Vert_{L^2(\partial E^*)}.$$\end{lemma}

\subsection{Control of the remainder term to obtain \eqref{Eq1}} We recall that the remainder term is given by 
\begin{equation}\bold R(t;\Phi):=\bold{R}_0(t,\Phi)+\bold{R}_1(t,\Phi)+\bold{R}_2(t,\Phi),\end{equation} with 
\begin{equation}
\begin{cases}\displaystyle
\bold{R}_0(t,\Phi)&=\displaystyle-2\int_{\partial E^*}\left\{J_{\Sigma,t}\hat u'_t\hat u_t-u_0'u_0\right\}\left(\Phi\cdot\nu\right),\\
\\\bold{R}_1(t,\Phi)&=\displaystyle-2\int_{\partial E^*}\left\{J_{\Sigma,t}\hat u_t\frac{\partial \hat u_t}{\partial \nu}-u_0\frac{\partial u_0}{\partial \nu}\right\}\left(\Phi\cdot\nu\right)^2,\\
\\\bold{R}_2(t,\Phi)&=\displaystyle\int_{\partial E^*}J_{\Sigma,t} \left\{\hat H_t (\hat u_t^2-\left.u_{0}^2\right|_{\partial E^*})\right\}\left(\Phi\cdot\nu\right)^2.
\end{cases}
\end{equation}
The goal of this subsection is to prove the following result (i.e. the validity of \eqref{Eq1}):
\begin{proposition}\label{Pr:Con}
There exists $p\in (d;+\infty)$, $\eta_p>0$ (given by Proposition \ref{Pr:Continuity}), a modulus of continuity $\omega:\R_+^*\to \R_+$ and a constant $\mathscr M>0$ such that, for any $\Phi \in W^{1,\infty}\cap W^{2,p}(\O;\R^d)$ satisfying $\Vert \Phi\Vert_{W^{2,p}}\leq \eta_p$ and for any $t\in (0;1)$,
\begin{equation}
\vert \bold R(t,\Phi)\vert \leq \mathscr M \Vert \Phi \cdot \nu \Vert_{L^2(\partial E^*)}^2\omega\left(\Vert \Phi\Vert_{W^{2,p}}\right).
\end{equation}
\end{proposition}
\begin{proof}[Proof of Proposition \ref{Pr:Con}]
We control the three terms $\bold R_k(t,\Phi)$, $k=0,1,2$, separately. We assume that $p$ is large enough to ensure that the embeddings $W^{2,p}\hookrightarrow \mathscr C^1$, $W^{1,p}\hookrightarrow \mathscr C^0$ are bounded.
\def\intpe{{\int_{\partial E^*}}}
\begin{enumerate}
\item \underline{Control of $\bold R_0(t,\Phi)$:} Let $\delta_0>0$ be given and let $\eta_0''$ be given by Proposition \ref{PR3}. We assume that $\Vert \Phi \Vert_{\mathscr C^1}\leq \eta_0''$, which is possible provided $\Vert \Phi\Vert_{W^{2,p}}$ is small enough.

From Estimates \eqref{Eq:Jac}, Proposition \ref{PR2}, Proposition \ref{PR3} and the Cauchy-Schwarz inequality, we obtain the existence of a constant $\mathcal M_0>0$ such that 
\begin{equation}\label{ControlR0}
\vert \bold R_0(t,\Phi)\vert \leq \mathcal M_0\left(\Vert \Phi \cdot \nu \Vert_{W^{1,\infty}(\partial E^*)}+\delta_0+\Vert \Phi\Vert_{W^{2,p}}\right)\Vert \Phi \cdot \nu \Vert_{L^2(\partial E^*)}^2.
\end{equation}
As a consequence, there exists a modulus of continuity $\omega_0$, a constant $\mathcal M_0$ and a parameter $\eta_{0,0}>0$ such that, for any $\Phi$ satisfying 
$$\Vert \Phi \Vert_{W^{2,p}}\leq \eta_{0,0}$$ there holds, for any $t\in (0;1)$,
\begin{equation}
\vert \bold R_0(t,\Phi)\vert \leq \mathcal M_0 \omega_0\left(\Vert \Phi \Vert_{W^{2,p}}\right)\Vert \Phi \cdot \nu \Vert _{L^2(\partial E^*)}^2.\end{equation}

\item \underline{Control of $\bold R_1(t,\Phi)$:} Let $\delta>0$ and $\eta>0$ be given by Proposition \ref{Pr:C1}. From Proposition \ref{PR2}, there exists a constant $\mathcal M_1>0$ such that for any $\Phi$ satisfying 
$$\Vert \Phi \Vert_{W^{2,p}}\leq \eta_{0,0}$$ there holds, for any $t\in (0;1)$,
\begin{equation}\label{ControlR1}
\vert \bold R_1(t,\Phi)\vert \leq\mathcal M_1\left(\Vert \Phi \cdot \nu \Vert_{W^{1,\infty}(\partial E^*)}+\delta\right)\Vert \Phi \cdot \nu \Vert_{L^2(\partial E^*)}^2.
\end{equation}
As a consequence, there exists a modulus of continuity $\omega_1$, a constant $\mathcal M_1$ and a parameter $\eta_{0,1}>0$ such that, for any $\Phi$ satisfying 
$$\Vert \Phi \Vert_{W^{2,p}}\leq \eta_{0,1}$$ there holds, for any $t\in (0;1)$,
\begin{equation}
\vert \bold R_1(t,\Phi)\vert \leq \mathcal M_1\omega_1\left(\Vert \Phi \Vert_{W^{2,p}}\right)\Vert \Phi \cdot \nu \Vert _{L^2(\partial E^*)}^2.\end{equation}
\item \underline{Control of $\bold R_2(t,\Phi)$:} From H\"{o}lder's inequality, for any $q,r\in (1;+\infty)$ such that
\begin{equation}\label{Eq:Coeff}\frac1r+\frac1q+\frac12=1\end{equation} there holds
\begin{equation}
\vert \bold R_2(t,\Phi)\vert \leq \Vert J_{\Sigma,t}\Vert_{L^\infty}\Vert \Phi \cdot \nu \Vert_{L^\infty}\Vert \hat H_t\Vert_{L^r}\Vert \hat u_t-u_0\Vert_{L^q}\Vert \Phi \cdot \nu \Vert_{L^2}.
\end{equation}
We now choose $q\in \left(2;2\frac{d-1}{d-2}\right)$, fix the corresponding $r$ solution to \eqref{Eq:Coeff}, the corresponding $\eta_p$ given by Lemma \ref{Le:Crucial} and apply Estimate \ref{Eq:Jac} and Lemma \ref{Le:Crucial} to obtain the existence of a constant $\mathcal M_2>0$ such that, for any $\Phi$ satisfying $\vert \Phi \Vert_{W^{2,p}}\leq \eta_p$ and any $t\in (0;1)$, there holds
\begin{equation} 
\vert \bold R_2(t,\Phi)\vert\leq \mathcal M_2 \Vert \Phi \cdot \nu \Vert_{L^\infty} \Vert \Phi\cdot \nu \Vert_{L^2}^2,\end{equation} and define the modulus of continuity $\omega_2\left(\Vert \Phi \Vert_{W^{2,p}}\right):=\Vert \Phi \Vert_{W^{2,p}}\geq C_0\Vert \Phi\Vert_{L^\infty}$ for some constant $C_0$ to obtain 
\begin{equation} 
\vert \bold R_2(t,\Phi)\vert\leq \mathcal M_2'\omega_2\left(\Vert \Phi \Vert_{W^{2,p}}\right)\Vert \Phi\cdot \nu \Vert_{L^2}^2.\end{equation}
\end{enumerate}
Summing these three contributions yields the required estimate \eqref{Eq1}, and we obtain the quantitative inequality for normal deformations of optimal spectral sets.
\end{proof}

As a consequence of that local quantitative inequality, we get the following Proposition:
\begin{proposition}\label{Pr:Isolated}Let $\O$ satisfies \eqref{A1}-\eqref{A2}. Optimal spectral sets are isolated in the following sense: there exists $\beta>0$ such that
\begin{equation}\label{Isol}\forall (E_1^*\,, E_2^*)\in (\mathcal E^*)^2\,, E_1^*\neq E_2^*\Rightarrow \operatorname{dist}_{L^1}(E_1^*\,, E_2^*)\geq \beta.\end{equation}

\end{proposition}

\begin{proof}[Proof of Proposition \ref{Pr:Isolated}]

We argue by contradiction and assume \eqref{Isol} does not hold. In particular, there exist two sequences $\{E_{i,k}\}_{k\in \N}\in (\mathcal E^*)^\N$, $i=1,2$ such that
$$\forall k \in \N \,, E_{1,k}\neq E_{2,k}\,, \operatorname{dist}_{\mathcal H}(E_{1,k},E_{2,k})\underset{k\to \infty}\rightarrow 0.$$
Since the perimeter of optimal spectral sets are uniformly bounded, there exists a $L^1$-strong limit  $E_1$ of the sequence $\{E_{1,k}\}_{k\in \N}$. As a consequence, we have 
\begin{equation}\label{Eq:Cov}\operatorname{dist}_{L^1}(E_1,E_{2,k})\underset{k\to \infty}\rightarrow 0.\end{equation}
We now recall that, for any $k\in \N$,  we have 
$$E_{2,k}=\{u_{E_{2,k}}=\mu(E_{2,k})\}$$ for some $\mu(E_{2,k})>0$, and, from Lemma \ref{Le:Regularity} and \eqref{Eq:Cov} we have, for any $s\in (0;1)$ and any $p\in (1;+\infty)$, 
\begin{equation}\label{Eq:CvEf}u_{E_{2,k}}\underset{k\to \infty}\rightarrow u_{E_1}\text{ in }\mathscr C^{1,s}(\O) \text{ and in }W^{2,p}(\O).\end{equation}
As a consequence, let us prove:
\begin{gather*}\label{Defor}\tag{\textbf{Def}}
\text{ For $k$ large enough,  $E_{2,k}$ is a deformation of $E_1$: $E_{2,k}=\Phi_{2,k}(E_1)$, and, }\\ \text{for any $p\in (1;+\infty)$, $\Vert \Phi_{2,k}\Vert_{W^{2,p}}\underset{k\to +\infty}\rightarrow 0.$}\end{gather*}

\begin{proof}[Proof of \eqref{Defor}]
First of all, \eqref{Eq:CvEf} and Lemma \ref{Le:Hopf} imply that $E_{2,k}\underset{k\rightarrow +\infty}\rightarrow E_1$ in the $L^1$ As a consequence, there exists $k_1>0$ such that, for any $k\geq k_1$, $E_{2,k}$ is a graph over $E_1$: to prove this, we argue by contradiction; if this were not the case then there would exist a sequence $\{x_k\}_{k\in \N}\in \left( \partial E_1\right)^\N$ and two sequences $\{t_{i,k}\}_{k\in \N}\in \R_+^*$, $i=1,2$ such that
\begin{itemize}
\item For any $k\in \N$, $t_{1,k}\neq t_{2,k}$
\item For $i=1,2$, 
$$t_{i,k} \underset{k\rightarrow +\infty}\rightarrow 0.$$
\item For $i=1,2$, for any $k\in \N$, 
$$x_k+t_{i,k}\nu (x_k)\in \partial E_{2,k}$$ where $\nu (x_k)$ is the normal to $\partial E_1$ at $x_k$.
\end{itemize}
By the intermediate value Theorem, this yields the existence of $t_k\in (\min_{i=1,2}t_{i,k};\max_{i=1,2}t_{i,k})$  (so that $t_k \underset{k\to \infty}\to 0$) such that
$$\langle \n u_{E_{2,k}}(x_k+t_k \nu(x_k))\,, \nu (x_k)\rangle=0.$$ 
Passing to the limit in this equation thanks to \eqref{Eq:CvEf} yields
$$\frac{\partial u_{E_1}}{\partial \nu}(x_k)=0,$$ which contradicts Lemma \ref{Le:Hopf}. 
\end{proof}

We can hence write $E_{2,k}=\Phi_{2,k}(E_1)$ for some vector field $\Phi_{2,k}$, for any $k$ large enough. The $W^{2,p}$-convergence to zero is immediately implied by the $W^{2,p}$ convergence of eigenfunctions and the non-degeneracy of the level sets.

As a consequence of the local quantitative inequality for graphs, we obtain, for $k$ large enough,
$$0=\overline \lambda-\overline \lambda=\lambda(E_{2,k})-\lambda(E_1)\geq \alpha |E_1\Delta E_{2,k}|^2>0,$$ which is the required contradiction.
\end{proof}

\subsection{Second step: an auxiliary problem}
Let us consider a parameter $\delta>0$ and the admissible class 
$$\mathcal I(\delta):=\left\{V \in \mathcal M\,, \operatorname{dist}_{L^1}(V,\mathcal I^*)=\delta\right\}.$$
We define the variational problem
\begin{equation}\label{OC:Delta}\tag{$P_\delta$}
\inf_{V\in \mathcal I(\delta)}\lambda(V).\end{equation} We begin with the following Proposition:
\begin{proposition}\label{Pr:Exist}
The variational problem \eqref{OC:Delta} has a solution $V_\delta$.
\end{proposition}
\begin{proof}[Proof of Proposition \ref{Pr:Exist}]
Consider a minimising sequence $\{V_k\}_{k\in \N}\in \mathcal I(\delta)^\N$ for the problem \eqref{OC:Delta}. Let $V_\infty$ be a weak $L^\infty-*$ closure point of the sequence. Since $\mathcal M$ is closed for this convergence, $V_\infty \in \mathcal M$. Lemma \ref{Le:LSCeigenvalue} ensures that
$$\lambda(V_k)\underset{k\to \infty}\rightarrow \lambda(V_\infty).$$ It remains to check that $V_\infty \in \mathcal I(\delta).$

 To prove this, we proceed as follows: we observe that, for any optimal set $E^*\in \mathcal E^*$ and $k\in \N$ we have
 $$\int_\O \vert V_k-\mathds 1_{E^*}\vert\geq \operatorname{dist}_{L^1}(V_k,\mathcal I^*)=\delta.$$ Define $h_k:=V_k-\mathds 1_{E^*}$. Since $V_k\in \mathcal M$, we have 
 $$\int_\O h_k=0\,,\quad  h_k\leq 0\text{ in }E^*\,,\quad h_k\geq 0 \text{ in }(E^*)^c.$$ Combining this with 
 $$\int_\O |h_k|=\int_\O \vert V_k-\mathds 1_{E^*}\vert\geq \delta$$ we get
 $$-\int_{E^*} h_k=\int_{(E^*)^c}h_k\geq \frac{\delta}2.$$
 Defining $h_\infty:=V_\infty-\mathds 1_{E^*}$ we have 
 $$h_k\underset{k\to \infty}\rightharpoonup h_\infty \text{ weakly in }L^\infty-*(E^*).$$ Since
 $$h_k\leq 0\text{ in }E^*\,, \quad \int_{E^*}|h_k|=-\int_{E^*}h_k$$we obtain
 $$h_\infty\leq 0\text{ in } E^*\,, \int_{E^*}|h_\infty|=-\int_{E^*}h_\infty\geq \frac{\delta}2.$$ Doing the same computations on $(E^*)^c$ we obtain $\int_{(E^*)^c}|h_\infty|=\int_{(E^*)^c} h_\infty\geq \frac{\delta}2$ and $h_\infty\geq 0$ in $(E^*)^c$, so that 
 $$\int_{\O}\vert V_\infty-\mathds1_{E^*}\vert \geq \delta.$$ Hence, 
 $$\operatorname{dist}_{L^1}(V_\infty,\mathcal I^*)\geq \delta.$$
 
 It remains to prove that $\operatorname{dist}_{L^1}(V_\infty, \mathcal I^*)=\delta.$ To do so, we first note that for any $k\in \N^*$, there exists $E^*_k\in \mathcal I^*$ such that 
 $$\int_\O |V_k-\mathds 1_{E_*^k}|=\delta,$$ where $\{V_k\}_{k\in \N}$ is the same minimising sequence as before. The existence of $E_k^*$ follows from the strong $L^1$ compactness of $\mathcal E^*$ guaranteed by the uniform perimeter bound.
 
  Following the same line of reasoning as in the previous step we define, for every $k\in \N$, $j_k:=V_k-\mathds 1_{E_k^*}$ and we have, in a similar way
 \begin{equation}\label{1}j_k\leq 0\text{ in }E_k^*\,, j_k\geq 0\text{ in }(E_k^*)^c\,, \int_\O j_k=0\,, \int_\O\vert j_k\vert=\delta\end{equation} and 
 \begin{equation}\label{2}-\int_{E_k^*}j_k=-\int_\O \mathds 1_{E_k^*}j_k=\int_{(E_k^*)^c}j_k=\int_{\O} \mathds 1_{(E_k^*)^c}j_k=\frac\delta2.\end{equation}
 
  Let $V_\infty^*$ be a weak $L^\infty-*$ closure point of the sequence $\{\mathds 1_{E_k^*}\}_{k\in \N}$.  From Lemma \ref{Le:LSCeigenvalue}, $V_\infty^*$ is a minimiser of \eqref{Eq:E1}. Since every solution of \eqref{Eq:E1} is a bang-bang function, there exists an optimal spectral set $E_\infty^*\in \mathcal I^*$ such that $V_\infty^*=\mathds 1_{E_\infty^*}$. Since bang-bang functions are extreme points of $\mathcal M$ it follows from \cite[Proposition 2.2.1]{HenrotPierre} that $\{m_k\}_{k\in \N}$ converges strongly in $L^1(\O)$ to $V_\infty^*$, so that passing to the limit in Equations \ref{1} and \ref{2} gives
  $$\int_\O |V_\infty-\mathds 1_{E_\infty^*}|=\delta.$$
  Hence, $V_\infty \in \mathcal I(\delta)$, and the proof is concluded.
  
\end{proof}
Throughout the rest of this paragraph, for any $\delta>0$, the notation $V_\delta$ stands for a solution of \eqref{OC:Delta}. The rest of this paragraph is devoted to the proof of the following Lemma, which provides a helpful reduction of Theorem \ref{Th:Quanti}:
\begin{lemma}\label{Le:DeltaSmall}
The conclusion of Theorem \ref{Th:Quanti} is equivalent to 
\begin{equation}\label{DeltaSmall} 
\underset{\overline{\delta \to 0}}\lim \frac{\lambda (V_\delta)-\overline \lambda}{\delta^2}>0.\end{equation}

\end{lemma}

\begin{proof}[Proof of Lemma \ref{Le:DeltaSmall}]
It is clear that the quantitative inequality of Theorem \ref{Th:Quanti} implies \eqref{DeltaSmall}. Conversely, assume that \eqref{DeltaSmall} holds. Let us prove that the conclusion of Theorem \ref{Th:Quanti} in turn holds. To do so,  consider a minimising sequence $\{V_k\}_{k\in \N}\in \left(\mathcal M\backslash \mathcal I^* \right)^\N$ for the functional 
$$\mathcal G:\mathcal M\backslash \mathcal I^*\ni V \mapsto \frac{\lambda(V)-\overline \lambda}{\operatorname{dist}_{L^1}(V,\mathcal I^*)^2}.$$
Two cases can be distinguished:
\begin{enumerate}
\item There exists a subsequence of $\{V_k\}_{k\in \N}$ such that (with a slight abuse of notation we assume that the entire sequence satisfies the property)
$$\operatorname{dist}_{L^1}(V_k,\mathcal I^*)\underset{k\to \infty}\rightarrow \epsilon^*>0.$$ In that case, we can extract a subsequence of $\{V_k\}_{k\in \N}$ that weakly converges to $V_\infty$. Reasoning with the same arguments as in the proof of Proposition \ref{Pr:Exist}, $V_\infty$ satisfies $$\operatorname{dist}_{L^1}(V_\infty,\mathcal I^*)=\epsilon^*$$ and so $\lambda(V_\infty)>\overline \lambda$. As such, 
$$\inf_{V\in \mathcal M\backslash \mathcal I^*} \mathcal G=\underset{k\to \infty}\lim \mathcal G(V_k)=\mathcal G(V_\infty)>0.$$
\item The second case is 
$$\operatorname{dist}_{L^1}(V_k, \mathcal I^*)\underset{k\to \infty}\rightarrow 0.$$ In that case, defining $\delta_k:=\operatorname{dist}_{L^1}(V_k, \mathcal I^*)$ we see that 
$$\mathcal G(V_k)\geq \frac{\lambda(V_{\delta_k})-\overline \lambda}{\delta_k^2},$$ so that \eqref{DeltaSmall} implies the conclusion.
\end{enumerate}

\end{proof}

Hence our main focus is on proving \eqref{DeltaSmall}, which will be the final step of this proof.

\subsection{Third step: A quantitative bathtub principle}

We prove in this Section a quantitative version of the bathtub principle. Let us consider, for $\overline s \in (0;1)$, a $\mathscr C^{1,s}$ function $f:\O\to \R$ such that 
there exists a unique $\mu(f;V_0)\in \R$ such that 
$$|\{f> \mu(f;V_0)\}|=|\{f\geq \mu(f;V_0)\}|=V_0$$ and define $\O(f;V_0):=\{f> \mu(f;V_0)\}$.  

In particular, the bathtub principle states that
\begin{equation}\label{bathtub}
V_{f;V_0}:=\mathds 1_{\Omega(f;V_0)}\text{ is the unique solution of }\sup_{V \in \mathcal M}\int_\O fV.\end{equation}

We need a regularity Assumption. We say that $f$ satisfies \eqref{H1} at the level $\mu(\O;V_0)$ if
\begin{equation}\tag{$\bold H_1$}\label{H1}\omega(f;V_0):=\inf\left(-\left.\frac{\partial f}{\partial \nu}\right|_{\partial \O(f;V_0)}\right)>0.\end{equation}

The goal of the next proposition is to give a quantitative version of \eqref{bathtub}.

\begin{proposition}\label{Pr:Bathtub}
There exists a constant $c_1>0$ depending on $\Vert f\Vert_{\mathscr C^{1,s}}$ such that the following holds:

For any $f$ satisfying \eqref{H1} at the level $V_0\in (0;|\O|)$, for any $V\in \mathcal M$, 
\begin{equation}\label{Eq:Bathtub}\int_\O V f\leq \int_\O V_{f;V_0}f-c_1\frac{\operatorname{Per}(\B_0)}{\operatorname{Per}(\O(f;V_0))} \omega(f;V_0)\Vert V-V_{\Omega(f;V_0)}\Vert_{L^1}^2.\end{equation}

\end{proposition}

\begin{proof}[Proof of Proposition \ref{Pr:Bathtub}]
Let us drop the subscript $V_0$ and write 
$$V_f:=V_{f;V_0}=\mathds 1_{f> \mu(f;V_0)}.$$
 Let us consider, for any $\delta>0$, the class
 $$\mathcal M(\delta):=\left\{ V \in \mathcal M\,, \Vert V-V_f\Vert_{L^1}=\delta\right\}.$$
 Proceeding as in Proposition \ref{Pr:Exist}, the optimisation problem
 $$\sup_{V\in \mathcal M(\delta)}\int_\O f V$$ has a solution.  Let us write this solution $\mathscr V_\delta$. From the same reasoning as in Lemma \ref{Le:DeltaSmall}, Equation \eqref{Eq:Bathtub} is equivalent to 
 \begin{equation}\label{Eq:BathtubDelta}
 \underset{\delta \to 0}{\overline \lim} \frac{\int_\O f\left(\mathscr V_\delta-V_f\right)}{\delta^2}<-c_1 \omega(f;V_0)\frac{\operatorname{Per}(\B_0)}{\operatorname{Per}(\O(f;V_0))} .
 \end{equation}

However, from the bathtub principle and from Assumption \eqref{H1} there exist, for any $\delta>0$ small enough, $\eta_\delta^{(1)}\,, \eta_\delta^{(2)}>0$ such that 
\begin{itemize}
\item $$\mathscr V_\delta=\mathds 1_{\{f\geq \mu(f;V_0)+\eta_\delta^{(1)}\}}+\mathds 1_{\{\mu(f;V_0)>f>\mu(f;V_0)-\eta_{\delta}^{(2)}\}},$$
\item $$\eta_\delta^{(i)}\underset{\delta \to 0}\rightarrow 0\,, i=1,2,$$
\item $$\vert \{\mu(f,V_0)+\eta_\delta^{(1)}>f>\mu(f;V_0)-\eta_{\delta}^{(2)}\}|=\delta,$$
\item$$\left\vert\left \{\mu(f,V_0)+\eta_\delta^{(1)}>f>\mu(f;V_0)\right\}\right\vert=\left\vert\left \{\mu(f,V_0)>f>\mu(f;V_0)-\eta_{\delta}^{(2)}\right\}\right\vert=\frac\delta2.$$
\end{itemize}
Here, Assumption \eqref{H1} is just used to ensure that the level sets $\{f=\mu(f;V_0)\pm \e\}$ have zero measure for $\e$ small enough.

Using this description of $\mathscr V_\delta$ we can now proceed to a Schwarz rearrangement; we refer to \cite{Almgren1989,Kesavan} for a full presentation of the Schwarz rearrangement, but let us recall that the Schwarz rearrangement of a bounded function $\psi:\O\to \R$ is the (unique) non-increasing radially symmetric function $\psi^*:\B^*\to \R$, where $|\B^*|=|\O|$ such that 
$$\forall t \in \R\,, \left\vert \{\psi\geq t\}\right\vert=\left\vert \{\psi^*\geq t\}\right\vert=\mathscr Z_\psi(t)=\mathscr Z_{\psi^*(t)},$$ where $\mathscr Z_g$ is the distribution function of a function $g$.

Let $\B_0=\mathbb B(0;r_0)$ be the unique centered ball of volume $V_0$ and $\B$ be the centered ball of volume $\O$. Since $f$ and $f^*$ have the same distribution functions, it follows from the co-area formula that 

$$\int_{\partial E} \frac1{\left|\frac{\partial f}{\partial \nu}\right|}=-\mathscr Z_f'(\mu(f;V_0))=-\mathscr Z_{f^*}'(\mu(f;V_0))=\int_{\partial \B_0}\frac1{\left|\frac{\partial f^*}{\partial \nu}\right|}$$ whenever $|\n f|>0$ on the level set $\{f=\mu(f;V_0)\}$, which is the case here by assumption \cite[Proof of Theorem 2.2.3]{Kesavan}. Since $f^*$ is radially symmetric, we write $\left|\frac{\partial f^*}{\partial \nu}\right|_{\partial \B_0}$ for the common value of that quantity on $\partial \B_0$ and obtain the bound
\begin{equation}\label{Eq:BoundSchwarz}
\left|\frac{\partial f^*}{\partial \nu}\right|_{\B_0}\geq \omega(f;V_0) \frac{\operatorname{Per}(\B_0)}{\operatorname{Per}(\O(f;V_0))}.
\end{equation}

We then consider 
$$\int_{\B^*} f^*\left(\mathscr V_\delta^*-V_f^*\right).$$ By equimeasurability of the rearrangements, we have
$$\int_{\B^*} f^*\left(\mathscr V_\delta^*-V_f^*\right)=\int_\O f\left(\mathscr V_\delta-V_f\right).$$
Furthermore, $\mathscr V_\delta^*-V_f^*$ satisfies, 
\begin{equation}\mathscr V_\delta^*-V_f^*=-\mathds 1_{\{r_0-r_\delta^{(1)}<r<r_0\}}+\mathds1_{\{r_0<r<r_0+r_\delta^{(2)}\}},\end{equation}
and $r_{\delta}^{(i)}\underset{\delta \to 0}\rightarrow 0$ for $i=1,2$. Here $r_{\delta}^{(1)}$ is chosen so that 
\begin{equation}\label{Eq:rdelta}\left\vert \{r_0-r_\delta^{(1)}\leq r\leq r_0\}\right\vert=\left\vert \{r_0\leq r\leq r_0+r_\delta^{(2)}\}\right\vert=\frac\delta2.\end{equation} 
\def\rdu{{r_\delta^{(1)}}}\def\rdd{{r_\delta^{(2)}}}

Explicit computations show that
\begin{equation}\label{Eq:RD}
\rdu\,, \rdd\underset{\delta \to 0}\sim \delta.\end{equation}

As a consequence, we have, by a radial change of variables, 
\begin{align*}
\int_{\B} f^*\left(\mathscr V_\delta^*-V_f^*\right)=&(2\pi)^{d-1}\left\{-\int_{r_0-r_\delta^{(1)}}^{r_0}r^{d-1}f^*(r)dr\right.
\\&\left. +\int_{r_0}^{r_0+r_\delta^{(2)}}r^{d-1}f^*(r)dr\right\}.
\end{align*}
We now apply a Taylor expansion of $f^*$ at $r=r_0$ and write
$$f^*(r_0+\e)=f(r_0)+\e (f^*)'(r_0)+\underset{\e \to 0}o(\e).$$ The remainder $\underset{\e \to 0}o(\e)$ can be written as $\underset{\e \to 0}O(\Vert f\Vert_{\mathscr C^{1,s}}(\e^{1+s}))$ where $O$ is uniform in $f$.
Thus we have
\begin{align*}
\int_{r_0-r_\delta^{(1)}}^{r_0}r^{d-1}f^*(r)dr=&\int_{r_0-r_\delta^{(1)}}^{r_0} r^{d-1}\left(f^*(r_0)+(r-r_0)(f^*)'(r_0)+\underset{r\to r_0}o(r-r_0)\right)dr
\\=&+f^*(r_0)\int_{r_0-r_\delta^{(1)}}^{r_0} r^ddr
\\&+(f^*)'(r_0)\int_{r_0-\rdu}^{r_0}r^{d-1}(r-r_0)dr
\\&+\underset{r\to r_0}o\left(\int_{r_0-\rdu}^{r_0}r^{d-1}(r-r_0)dr\right).
\end{align*}
However:

\begin{align*}
\int_{r_0-\rdu}^{r_0}r^{d-1}(r-r_0)dr=-r_0^{d-1}(\rdu)^2+\underset{\delta \to 0}o((\rdu)^2)\underset{\delta \to 0}\sim -C\delta^2,
\end{align*}
where the last equivalence comes from \eqref{Eq:RD}.

In the same way we obtain
\begin{align*}
\int_{r_0}^{r_0+r_\delta^{(2)}}r^{d-1}f^*(r)d=&f^*(r_0)\int_{r_0}^{r_0+\rdd} r^ddr
\\&+(f^*)'(r_0)C\delta^2
\\&+\underset{r\to r_0}o\left(\delta^2\right)
\end{align*}
for the same constant $C$.

Summing these contributions gives
\begin{align*}
\frac1{(2\pi)^d}\int_{\B^*} f^*\left(\mathscr V_\delta^*-V_f^*\right)=&f^*(r_0)\left\{\int_{r_0}^{r_0+\rdd} r^ddr-\int_{r_0-r_\delta^{(1)}}^{r_0} r^ddr\right\}\tag{$\bold I$}
\\&+2C(f^*)'(r_0)\delta^2
\\&+\underset{\delta \to0}o(\delta^2).
\end{align*}
However, $(\bold I)=0$ by \eqref{Eq:rdelta}, and we are thus left with 
$$\int_{\B^*} f^*\left(\mathscr V_\delta^*-V_f^*\right)\underset{\mu \to 0}\sim C' (f^*)'(r_0)\delta^2.$$
From \eqref{Eq:BoundSchwarz} and the fact that $f^*$ is non-increasing we obtain 
$$(f^*)'(r_0)\leq - \omega(f;V_0) \frac{\operatorname{Per}(\B_0)}{\operatorname{Per}(\O(f;V_0))},$$
and the conclusion thus follows.

\end{proof}


\subsection{Step 4: Proof of Theorem \ref{Th:Quanti}}
\begin{proof}[Proof of Theorem \ref{Th:Quanti}]
Let us argue by contradiction and assume that \eqref{Eq:Quanti} does not hold. From Lemma \ref{Le:DeltaSmall}, this is equivalent to the existence of a sequence $\{\delta_k\}_{k\in \N}\in (\R_+^*)^\N$ such that 
\begin{itemize}
\item $\delta_k \underset{k\to \infty}\rightarrow 0$,
\item \begin{equation}\label{Eq:Contrad}
\frac{\lambda(V_{\delta_k)}-\overline \lambda}{\delta_k^2}\underset{k\to \infty}\longrightarrow 0.
\end{equation}\end{itemize}

Let us now consider a weak $L^\infty-*$ closure point $V_\infty$ of $\{V_k\}_{k\in \N}\in \mathcal M$. Since 
$$\operatorname{dist}_{L^1}(V_k,\mathcal I^*)=\delta_k,$$ $V_\infty\in \mathcal I^*$ is a solution of \eqref{Eq:E1}. Let us assume, with a slight abuse of notation, that the entire sequence converges to $V_\infty$. Since any solution of \eqref{Eq:E1} is the characteristic function of a set, there exists $E^*\in \mathcal E^*$ such that 
$V_\infty=\mathds 1_{E^*}.$ Since this is an extremal function in $\mathcal M$ it follows from \cite[Proposition 2.2.1]{HenrotPierre} that this convergence is strong in $L^1$:
$$V_k\underset{k\to \infty}{\xrightarrow{L^1}} V_\infty.$$

Let us now consider the sequence $\{V_k^*\}_{k\in \N}\in (\mathcal I^*)^\N$ such that, for any $k\in \N$, there holds
$$\operatorname{dist}_{L^1}(V_k^*,V_k)=\delta_k.$$ The existence of  such a $V_k^*$ follows from the direct method of the calculus of variations and the uniform perimeter bound in $\mathcal E^*$. Let us consider in the same way a strong $L^1$ closure point $W_\infty$ of $\{V_k^*\}_{k\in \N}$ (it is a strong closure point because of the uniform perimeter bound on $\mathcal E^*$). We claim that the only closure point is $W_\infty=V_\infty$, so that the entire sequence converges. This is a simple consequence of the fact that $W_\infty$ is a solution of \eqref{Eq:E1} and of the triangle inequality
$$\forall k\in \N\,, \Vert V_\infty-W_\infty\Vert_{L^1} \leq \Vert V_\infty-V_k\Vert_{L^1}+\Vert V_k-V_k^*\Vert_{L^1}+\Vert V_k^*-W_\infty\Vert_{L^1},$$ and each of the terms on the right hand side converges to 0 as $k\to \infty$.

We then claim that $V_\infty=V_k^*$ for $k$ large enough; this follows from Proposition \ref{Pr:Isolated} and once again from the triangle inequality. We thus have 
$$\forall k\in \N\,, \operatorname{dist}_{L^1}(V_\infty,V_k)=\delta_k.$$

Let us write $u_k:=u_{V_k}$ and $u^*:=u_{V_\infty}.$ Lemma \ref{Le:Regularity} ensures that, for every $s\in (0;1)$ and every $p\in (1;+\infty)$, we have 
\begin{equation}\label{Eq:CVC}u_k\underset{k\to \infty}{\xrightarrow{W^{2,p}(\O)\,, \mathscr C^{1,s}(\O)}}u^*.\end{equation}

Let us define for every $k\in \N$ the unique real number $\mu_k$ such that 
\begin{equation}
\left\vert \{u_k>\mu_k\}\right\vert=V_0.\end{equation}
In other words, we have, with the notations of Proposition \ref{Pr:Bathtub}, $\mu_k=\mu(u_k;V_0)$. Let us also define $\mu^*:=\mu(u^*;V_0)$. We define the affiliated super level sets
$$\forall k\in \N \,, E_k:=\Omega(u_k;V_0):=\{u_k> \mu_k\}\,, E^*:=\Omega(u^*;V_0):=\{u^*>\mu^*\}.$$ From the fact that $V_\infty:=\mathds 1_{E^*}$, \eqref{Eq:CVC} and Assumption \eqref{A2} we get that 
$$\mu_k\underset{k\to \infty}\rightarrow \mu^*$$ and, for any $k$ large enough
\begin{equation}\label{Eq:NDk}\inf_{\partial E_k}\left(-\frac{\partial u_k}{\partial \nu}\right)\geq\overline{\omega}:= \frac12\inf_{\partial E^*}\left(-\frac{\partial u^*}{\partial \nu}\right)>0.\end{equation} Up to extracting another subsequence, we assume that \eqref{Eq:NDk} is satisfied along the entire sequence.

Finally, using the same arguments as in the proof of Proposition \ref{Pr:Isolated}, for $k$ large enough,  $E_k$ is a graph over $E^*$ and $E_k$ converges to $E^*$ in  $\mathscr C^{1,s}$ and  $W^{2,p}$ for any $s\in [0;1)\,, p\in (1;+\infty).$ In other words, for any $k$ large enough, there exists a vector field $\Phi_k$ such that $E_k=\Phi_k(E^*)$ and that converges to $0$ in $W^{2,p}$ for every $p\in (2;+\infty)$. We sum up these information in the following claim:
\begin{equation}\label{Eq:EkEe}\text{ For $k$ large enough, there exists $\Phi_k$ such that }E_k=\Phi_k(E^*).  \text{  }\forall p\in [2;+\infty) \Vert \Phi_k\Vert_{W^{2,p}}\underset{k\to \infty}\rightarrow 0.\end{equation}

From the $\mathscr C^{1,s}$ convergence of level sets and the uniform bound on the perimeter of optimal spectral sets, we get that there exists $\overline{\operatorname{Per}}<\infty$ such that, for any  $k$ large enough
\begin{equation}\label{Eq:PerK}\operatorname{Per}(E_k)\leq \overline{\operatorname{Per}}.\end{equation}

\begin{enumerate}
\item \textbf{Using the bathtub principle:} 
We first replace $V_k$ by $\tilde V_k:=V_{u_k;V_0}=\mathds 1_{E_k}\in \mathcal M$. Using the Rayleigh quotient formulation of the eigenvalue \eqref{Eq:Rayleigh} we get
\begin{align}
\lambda(V_k)&=\int_\O |\n u_k|^2-\int_\O V_k u_k^2
\\&\geq \int_\O |\n u_k|^2-\int_\O \tilde V_k u_k^2\label{Eq:Aq}
\\&\geq \lambda(\tilde V_k),
\end{align}
and we can quantify \eqref{Eq:Aq} using the quantitative bathtub principle. From Proposition \ref{Pr:Bathtub}, Equation \eqref{Eq:NDk} and \eqref{Eq:PerK} we have, for a constant $c_1>0$ 
\begin{align*}
\int_{\O} V_k u_k^2-\int_\O \tilde V_k u_k^2\leq -c_1\frac{\operatorname{Per}(\B_0)}{\overline{\operatorname{Per}}}{\overline \omega}\Vert V_k-\tilde V_k\Vert_{L^1}^2.
\end{align*}
Setting $\overline c:=c_1\frac{\operatorname{Per}(\B_0)}{\overline{\operatorname{Per}}}{\overline \omega}$ we get 
\begin{equation}
\lambda(V_k)-\lambda(\tilde V_k)\geq \overline c \Vert V_k-\tilde V_k\Vert_{L^1}^2.
\end{equation}
\item \textbf{Using normal deformations:} We recall that $\lambda(\tilde V_k)=\lambda(E_k)$. From \eqref{Eq:EkEe} and Proposition \ref{Pr:LocalStabilityUn} we obtain that there exists $k_1\in \N$ such that
$$\forall k \in \N\,, k\geq k_1 \Rightarrow \lambda(E_k)-\lambda(E^*)\geq \underline \alpha \left\vert E_k\Delta E^*\right\vert^2=\underline \alpha \Vert \tilde V_k-V_\infty\Vert_{L^1}^2.$$
\item \textbf{Conclusion:}
Summing the two previous steps we obtain
\begin{align*}
\lambda(V_k)-\overline \lambda=&\lambda(V_k)-\lambda(\tilde V_k)
\\&+\lambda(\tilde V_k)-\overline \lambda
\\&\geq \overline c \Vert V_k-\tilde V_k\Vert_{L^1}^2
\\&+\underline \alpha \Vert \tilde V_k-V_\infty\Vert_{L^1}^2.
\end{align*}
Setting $\doubleunderline{\alpha}:=\min\left(\underline \alpha\,, \overline c\right)$ we thus have 
\begin{equation}\label{Almost}
\lambda(V_k)-\overline \lambda\geq \doubleunderline{\alpha}\left(\Vert V_k-\tilde V_k\Vert_{L^1}^2+\Vert \tilde V_k-V_\infty\Vert_{L^1}^2\right).
\end{equation}
Now, since for $k$ large enough
$\delta_k=\Vert V_k-V_\infty\Vert_{L^1}$ the triangle inequality implies, for $k$ large enough

\begin{equation}
\delta_k\leq \Vert  V_k-\tilde V_k\Vert_{L^1}+\Vert \tilde V_k-V_\infty\Vert_{L^1}.\end{equation}

Thus we can pick a subsequence of $\{V_k\}_{k\in \N}$ such that either 
$$\underset{k\to \infty}{\underline \lim}\frac{\Vert  V_k-\tilde V_k\Vert_{L^1}}{\delta_k}=\underline m>0$$ or

$$\underset{k\to \infty}{\underline \lim}\frac{\Vert  V_\infty-\tilde V_k\Vert_{L^1}}{\delta_k}=\overline m>0.$$ Setting $\doubleunderline m:=\max \left(\overline m\,, \underline m\right)$ and plugging this in \eqref{Almost} yields
$$\underset{k\to \infty}{\underline \lim}\frac{\lambda(V_k)-\lambda(V_\infty)}{\delta_k^2}\geq \doubleunderline \alpha \,\doubleunderline m>0,$$ which is in contradiction with \eqref{Eq:Contrad}. This concludes the proof.

\end{enumerate}

\end{proof}

\section{Proof of Theorem \ref{Th:Turnpike}}\label{Se:Turnpike}
\begin{proof}[Proof of Theorem \ref{Th:Turnpike}]
In order to prove the turnpike property, we will not use the optimality conditions but rather use a direct comparison argument. Let us consider any $V^*\in \mathcal I^*$ and consider the solution $\overline y$ of  \eqref{Eq:Main} associated with the static control $\mathcal V(t,\cdot)\equiv V^*(\cdot)$. Classical spectral decomposition arguments yield that
$$\overline y(t,x)=\sum_{k=1}^\infty\left(\int_\O y_0 \p_{k,V^*}\right) e^{-\lambda_k(V^*)t}\p_{k,V^*},$$
where the eigenvalues $\lambda_k(V^*)$ are the eigenvalues of $-\Delta -V^*$, which are ordered increasingly and are associated with $L^2$ normalized eigenfunctions $\p_{k,V^*}$ (and $\p_{1,V^*}=u_{V^*}$).   Since $\lambda(V^*)=\lambda_1(V^*)=\overline \lambda$ is a simple eigenvalue and since  the first associated eigenfunction $u_{V^*}=\p_{1,V^*}$ is positive and the initial condition of \eqref{Eq:Main} is non-negative and non-zero we can define 
$$A_0=\int_\O y_0  u_{V^*}>0$$ and we claim that there holds

\begin{equation}\label{63}\int_\O \overline y (T,x)\sim A_0e^{-\overline \lambda T}.\end{equation} Indeed, for every $k\in \N$, we have 
$$\left|\int_\O y_0 \p_{k,V^*}\right|\leq  \Vert y_0\Vert_{L^\infty}||\p_{k,V^*}||_{L^1}\leq \Vert y_0\Vert_{L^\infty}\Vert \p_{k,V^*}\Vert_{L^2}\sqrt{|\O|}=\sqrt{|\O|}\Vert y_0\Vert_{L^\infty}.$$ To obtain \eqref{63} one simply has to put $e^{-\overline \lambda T}$ in factor of the series.

We now consider the optimal control $\mathcal V_T^*$ associated with \eqref{Eq:PT}. We denote by $y_T^*$ the associated solution of  \eqref{Eq:Main}. We have, by multiplication of \eqref{Eq:Main} by $y_T^*$, integration by parts and by the Rayleigh quotient formulation of eigenvalues \eqref{Eq:Rayleigh},
\begin{align*}
\frac12\frac{\partial}{\partial t}\int_\O ({y_T^*})^2&=- \int_\O |\n {y_T^*}|^2+\int_\O \mathcal V_T^*({y_T^*})^2
\\&\leq -\lambda(\mathcal V_T^*(s,\cdot))\int_\O ({y_T^*})^2.
\end{align*}
With a slight abuse of notation, let us define 
$$\lambda(s):=\lambda_1\Big(\mathcal V_T^*(s,\cdot)\Big).$$
From the Gr\"{o}nwall Lemma, there holds, for some constant $B_0$ depending on $\Vert u_0\Vert_{L^2}$,
$$\int_\O (y_T^*)^2\leq\displaystyle  e^{-2\int_0^T \lambda_1(s)ds} B_0.$$

Since $\mathcal V_T^*$ solves \eqref{Eq:PT} we have
$$\int_\O \overline y(T,\cdot)\leq \int_\O y_T^*(T,x)dx\leq \sqrt{|\O|}\Vert y_T^*(T,\cdot)\Vert_{L^2}.$$  Hence we get, for some contant $\tilde B_0$
$$0< \int_\O \p_{1,V^*} y_0\leq\tilde  B_0e^{\int_0^T\left(\overline \lambda-\lambda_1(s)\right)ds}$$ or, alternatively,  there exists $M$ such that $$\int_0^T \left(\overline \lambda-\lambda_1\right)\geq -M.$$ 

Using Theorem \ref{Th:Quanti}, this writes as $$-M\leq \int_0^T\left(\overline \lambda -\lambda_1(s)\right)ds\leq -C\int_0^T \operatorname{dist}_{L^1}(\mathcal V_T^*(s,\cdot)\,, {\mathcal I^*})^2ds.$$
The conclusion follows immediately.

\end{proof}

\section{Conclusion}

In this article, we have obtained a quantitative result for a time evolving optimal control problems using shape optimisation tools to derive quantitative inequalities for scalar problems. In a series of future works, we plan on developing this approach to obtain several other interesting properties for parabolic and hyperbolic optimal control problems using mainly ideas coming from the sensitivity analysis of stationary elliptic problems.

\appendix

\section{Proof of technical results}\label{An:Technical}

\subsection{Proof of Lemma \ref{Le:Regularity}}
\begin{proof}[Proof of Lemma \ref{Le:Regularity}]
The weak formulation of the eigenfunction, the fact that it satisfies $\int_\O u_V^2=1$ and the constraint $||V||_{L^\infty}\leq 1$ immediately give, for any $V\in \mathcal M$, 
$$\int_\O |\n u_V|^2\leq \lambda(V)+1.$$
The Rayleigh quotient formulation of the eigenvalue \eqref{Eq:Rayleigh} and the constraint $V\geq 0$ give
$$\lambda(V)\leq  \lambda_1^D(\O),$$ where $\lambda_1^D(\O)$ the first Dirichlet eigenvalue of the domain $\O$. Hence for any $V\in \mathcal M$, 
$$\int_\O |\n u_V|^2\leq \lambda_1^D(\O)+1.$$

 As a consequence of the Poincar\'e inequality, there exists $Y_0$ such that, for any $V\in \mathcal M$ there holds
$$\Vert u_V\Vert_{W^{1,2}_0(\O)}\leq Y_0.$$ From the $W^{2,2}$-regularity for the Laplace operator, there exists $Y_1$ such that  for any $V\in \mathcal M$ there holds
\begin{equation}\label{Eq:W22}\Vert u_V\Vert_{W^{2,2}(\O)}\leq Y_1.\end{equation}
 
 We then apply a bootstrap argument as follows: 
 \begin{itemize}
 \item \underline{When $d=1,2,3$:} In this case, we can apply the classical Sobolev embedding $W^{k,p}(\O)\rightharpoonup \mathscr C^{r,\alpha}(\O)$ with $\frac{1}p-\frac{k}n=-\frac{r+\alpha}d$ in the case $k=p=2$, which yields the existence of a constant $Y_2$ such that for any $V\in \mathcal M$ there holds
$$\Vert u_V\Vert_{L^\infty(\O)}\leq \Vert u_V\Vert_{\mathscr C^{0,\alpha}(\O)}\leq Y_1.$$ From the $W^{2,p}$ regularity of the Laplace operator, for any $p\in (1;+\infty)$ there exists a constant $Y_3(p)$ such that for any $V\in \mathcal M$ there holds
$$\Vert u_V\Vert_{W^{2,p}(\O)}\leq Y_3(p).$$ Fixing first $s\in (0;1)$ and then $p=p(s)$ large enough that the embedding $W^{2,p}(\O)\rightharpoonup \mathscr C^{1,s}(\O)$ holds, we obtain the existence of a constant $\mathscr Y(s)$ such that  for any $V\in \mathcal M$ there holds
$$ \Vert u_V\Vert_{\mathscr C^{1,s}(\O)}\leq \mathscr Y(s).$$
\item \underline{When $d=4$:} In that case we apply the embedding $W^{2,2}(\O)\rightharpoonup L^q(\O)$ for any $q\in (2;+\infty).$ This gives, for any $q\in (2;+\infty)$, the existence of $Y_2(q)$ such that, for any $V\in \mathcal M$, 
$$\Vert u_V\Vert_{L^q(\O)}\leq Y_2(q).$$  From the $W^{2,q}$ regularity of the Laplace operator, for any $q\in (2;+\infty)$ there exists a constant $Y_3(q)$ such that for any $V\in \mathcal M$ there holds
$$\Vert u_V\Vert_{W^{2,q}(\O)}\leq Y_3(q).$$Fixing first $s\in (0;1)$ and then $q=q(s)$ large enough that the embedding $W^{2,q}(\O)\rightharpoonup \mathscr C^{1,s}(\O)$ holds, we obtain the existence of a constant $\mathscr Y(s)$ such that  for any $V\in \mathcal M$ there holds
$$ \Vert u_V\Vert_{\mathscr C^{1,s}(\O)}\leq \mathscr Y(s).$$
\item\underline{When $d>4$:} We start an iterative procedure. The uniform $W^{2,2}$ bound \eqref{Eq:W22} gives a $L^{q_1}(\O)$ uniform bound, with $q_1:=\frac{2d}{d-4}.$ This in turn, through the $W^{2,q_1}$ regularity for the Laplace operator, yields a uniform $W^{2,q_1}(\O)$ bound on the family $\{u_V\}_{V\in \mathcal M}$. If $2q_1>d$ (which corresponds to $d<8$) we can then apply the Sobolev embedding $W^{2,q_1}(\O)\rightharpoonup \mathscr C^{0,\alpha}(\O)$ for some $\alpha \in (0;1)$. We then conclude as in the previous cases. Otherwise we obtain a uniform $L^{q_2}(\O)$ bound with $q_2=\frac{d-8}{2d}$ and we then reiterate this procedure for as long as need, obtaining a sequence of uniform $L^{q_j}(\O)$ where $q_j=:\frac{r_j}{2d}$ is defined recursively via 
$$r_{j+1}=r_j-4\,, r_1=d.$$ We consider the first $j+1$ such that $r_{j+1}<0$ and apply the Sobolev embeddings to $q_j$. This gives an $L^\infty$ uniform bound on the eigenfunction, and we can then apply the same reasoning.
 \end{itemize}
 
 \end{proof}
\subsection{Proof of Lemma \ref{Le:LSCeigenvalue}}\label{An:LSC}
\begin{proof}[Proof of Lemma \ref{Le:LSCeigenvalue}]
Assume that the sequence $\{V_k\}_{k\in \N}\in \mathcal M^\N$ converges weakly to $V_\infty\in \mathcal M$. Define, for any $k\in \N$, $u_k$ as the eigenfunction associated with $V_k$ and $\lambda_k$ as the eigenvalue associated with it. From the Rayleigh quotient formulation \eqref{Eq:Rayleigh} of the eigenvalue, we immediately see that there exists a uniform upper bound on the sequence $\{\lambda_k\}_{k\in \N}$. Furthermore, denoting by $\lambda^D(\O)$ the first Dirichlet eigenvalue of $\O$, we also have 
$$\lambda_k\geq \lambda^D(\O)-1.$$ Hence there exists $\lambda_\infty$ such that, up to a subsequence, $\{\lambda_k\}_{k\in \N}$ converges to $\lambda_\infty$. We will show that $\lambda_\infty=\lambda(V_\infty)$, which will guarantee the convergence of the entire sequence.

To do so, we first observe that the sequence $\{u_k\}_{k\in \N}$ is uniformly bounded in $W^{1,2}_0(\O)$. By the Rellich-Kondrachov embedding theorem, there exists $u_\infty\in W^{1,2}_0(\O)$ such that, up to a subsequence, $\{u_k\}_{k\in \N}$ converges to $u_\infty$ weakly in $W^{1,2}_0(\O)$ and strongly in $L^2(\O)$. As a consequence we have 
$$\int_\O u_\infty^2=1\,, u_\infty \geq 0\text{ in } \O.$$ Passing to the limit in the weak formulation of the eigen-equation \eqref{Eq:Eig} proves that 
$$-\Delta u_\infty=\lambda_\infty u_\infty+V_\infty u_\infty$$ or, in other words, that $u_\infty$ is an eigenfunction associated with $V_\infty$. Since $u_\infty$ has constant sign, it follows that $\lambda_\infty$ is the first eigenvalue of $-\Delta -V_\infty$ and that $\lambda_\infty=\lambda(V_\infty)$. We then have, by lower weak semi-continuity of the norm
$$\underset{k\to \infty}{\underline \lim}\lambda(V_k)=\underset{k\to \infty}{\underline \lim}\left(\int_\O |\n u_k|^2-\int_\O V_k u_k^2\right)\geq \int_\O |\n u_\infty|^2-\int_\O V_\infty u_\infty^2=\lambda(V_\infty).$$ This concludes the proof.

\end{proof}
\bibliographystyle{abbrv}
\bibliography{BiblioMRBZ20201}

\end{document}